\documentclass[11pt,a4paper]{article}
\usepackage{caption2}%---change the caption of figures
\usepackage{CJK}
\usepackage{tabls}
\usepackage{bm}
\usepackage{multirow}
\usepackage{amssymb}
\usepackage{amsfonts}
\usepackage{mathrsfs}
\usepackage{empheq}
\usepackage{amsmath}
\usepackage{amsthm}
\usepackage{bbm}
\usepackage{graphicx}
\usepackage{color}
\usepackage{indentfirst}
\usepackage{hyperref}%pour PDFLaTex
\usepackage{float}
\usepackage{cases}
\usepackage[titletoc]{appendix}
\usepackage[colorinlistoftodos,english]{todonotes}
\usepackage{aligned-overset}
\usepackage{makecell}

\usepackage[latin1]{inputenc}
\usepackage[T1]{fontenc}
\usepackage{tikz}
\usetikzlibrary{shapes,arrows,calc,positioning}

	\tikzstyle{decision} = [diamond, draw, fill=blue!20, 
	text width=6em, text badly centered, node distance=3cm, inner sep=0pt]
	\tikzstyle{block} = [rectangle, draw, fill=blue!20, 
	text width=8em, text centered, rounded corners, minimum height=1em]
	\tikzstyle{block2} = [rectangle, draw, fill=yellow!20, 
	text width=9em, text centered, rounded corners, minimum height=4em]
	\tikzstyle{line} = [draw, -latex']
	\tikzstyle{cloud} = [draw, ellipse,fill=red!20, node distance=3cm,
	minimum height=4em]

\tikzstyle{intt}=[draw,text centered,minimum size=6em,text width=5.25cm,text height=0.34cm]
\tikzstyle{intl}=[draw,text centered,minimum size=2em,text width=2.75cm,text height=0.34cm]
\tikzstyle{int}=[draw,minimum size=2.5em,text centered,text width=3.5cm]
\tikzstyle{intg}=[draw,minimum size=3em,text centered,text width=6.cm]
\tikzstyle{sum}=[draw,shape=circle,inner sep=2pt,text centered,node distance=3.5cm]
\tikzstyle{summ}=[drawshape=circle,inner sep=4pt,text centered,node distance=3.cm]

\allowdisplaybreaks%

\def\e{\eqref}

\def\i1n{i=1,\cdots,n}
\def\j1n{j=1,\cdots,n}
\def\ij1n{i,j=1,\cdots,n}

\def\R{\mathbb R}

\def \i{\mathrm i}

 \numberwithin{equation}{section}
\theoremstyle{definition}

\def\R{{\bf R}}

\def\e{{\epsilon}}

\newtheorem{thm}{Theorem}[section]
\newtheorem{cor}{Corollary}[section]
\newtheorem{lem}{Lemma}[section]

\newtheorem{rem}{Remark}[section]

\begin{document}

\begin{CJK*}{GB}{gbsn}
\title{Sharp lifespan estimate for the compressible Euler system with critical time-dependent damping in $\R^2$}

%For each author, make a block with the following macros:

\author{Lv Cai \thanks{School of Mathematical Sciences, Fudan University, Shanghai 200433, China (18110180023@fudan.edu.cn).}
\and Ning-An Lai \thanks{School of Mathematical Sciences, Zhejiang Normal University, Jinhua 321004, China (ninganlai@zjnu.edu.cn).}
\and Wen-Ze Su \thanks{School of Mathematical Sciences, Fudan University, Shanghai 200433, China (19110180013@fudan.edu.cn).}}	

\maketitle

\begin{abstract}

This paper concerns the long time existence to the smooth solutions of the compressible Euler system with critical time dependent damping in $\R^2$. We establish the sharp lifespan estimate from below, with respect to the small parameter of the initial perturbation. For this end, the vector fields $\widehat{Z}$ (defined below) are used instead of the usual one $Z$, to get better decay for the linear error terms. This idea may also apply to the long time behavior study of nonlinear wave equations with time-dependent damping.

\end{abstract}


\section{Introduction}

\par
In this paper, we are concerned with the maximum existence time (lower bound of lifespan estimate) to the smooth solutions of the Cauchy problem of compressible Euler equations with time-dependent damping in $\R^2$:
\begin{equation} \label{Euler eqns}
	\begin{aligned}
		\begin{split}
			\left\lbrace
			\begin{array}{lr}
				\partial_{t} \rho + \nabla \cdot (\rho u)=0,~~~(t, x)\in [0, T]\times\R^2,\\
				\partial_{t} (\rho u) + \nabla \cdot (\rho u \otimes u ) + \nabla p = - \frac{\mu}{1+t} \rho u ,~(t, x)\in [0, T]\times\R^2,\\
				u(0,x) = \epsilon u_{0}(x), \ \rho (0,x) = \bar{\rho} + \epsilon \rho_{0}(x),~~~~~~x\in \R^2.
			\end{array}	
			\right. 
		\end{split}	
	\end{aligned}
\end{equation}
 The unknown functions are the strictly positive fluid density $ \rho: [0,T]  \times \R^{2} \rightarrow \R $, the velocity field $ u = ( u_{1}, u_{2} ): [0,T]  \times \R^{2} \rightarrow \R^{2} $, and the pressure $ p: [0,T]  \times \R^{2} \rightarrow \R $. The positive constant $\mu$ describes the effect of the damping. Moreover, we assume that the pressure $p$ satisfies the $\gamma$-law
\begin{equation} 
	p = p( \rho ) = \frac{1}{\gamma} \rho^{\gamma},
	\label{pressure law}
\end{equation}
where $\gamma > 1 $ is the adiabatic index.

\subsection{Review of the literature}
The compressible Euler equations with time-dependent damping term describe the evolution of a compressible fluid subject to both the conservation laws and the dissipation of energy due to damping. These equations are widely used in the study of fluid dynamics, and also have applications in various fields, including aerodynamics, astrophysics, and combustion. Specifically, the compressible Euler equations form the fundamental system governing the motion of compressible fluids. However, in realistic scenarios, the presence of various damping effects can significantly influence the behavior of the fluid motion. We try to give a comprehensive review of the literature on the following system:
\begin{equation} \label{1.2}
	\begin{aligned}
		\begin{split}
			\left\lbrace
			\begin{array}{lr}
				\partial_{t} \rho + \nabla \cdot (\rho u)=0, \ \ (t,x) \in [0,T] \times \R^{n} \\
				\partial_{t} (\rho u) + \nabla \cdot (\rho u \otimes u ) + \nabla p = - \frac{\mu}{(1+t)^{\lambda}} \rho u , \ \ (t,x) \in [0,T] \times \R^{n} \\
				u(0,x) = \epsilon u_{0}(x), \ \rho (0,x) = \bar{\rho} + \epsilon \rho_{0}(x) , \ \ x \in \R^{n}
			\end{array}	
			\right. 
		\end{split}	
	\end{aligned}
\end{equation}
where $ n $ is the spatial dimension, $ \frac{\mu}{(1+t)^{\lambda}} $ denotes the time-dependent frictional coefficient, with $ \mu \geq 0 $ and $ \lambda \geq 0 $ representing the strength and decay of the damping respectively.

\subsubsection{Euler system without damping}
Setting $ \mu =0 $, the system $ ( \ref{1.2} ) $ reduces to the classic Euler equations, a set of quasilinear hyperbolic equations, which is fundamental in fluid dynamics to describe the adiabatic and inviscid flow of an ideal fluid. Generally speaking, the compressible Euler equations develop shock waves in finite time for general initial data (see \cite{CH09, CH07, Lax64, Lax72, PS06, Sideris1985, Xin98} and the references therein). In the following, we will use $C,C_{1},C_{2}$ to denote
generic positive constants independent of $\epsilon$, the value of which may change from line to line.

\begin{table}[htb]
\centering
		\caption{Euler system without damping}
		\label{table:1}
	\begin{tabular}{|c|c|c|}
	\hline  \textbf{Author} & \textbf{Result} & \textbf{Remark} \\
	\hline  Sideris\cite{Sideris1985} & 3D, blow-up results for ideal fluid & \thead{with large data and \\ small initial perturbations} \\
	\hline  $^{*}$Sideris\cite{Sideris1991}\cite{Sideris1992}& 3D, $ \exp{(C_{1}\epsilon^{-1})} \leq T_{\epsilon} \leq \exp{(C_{2}\epsilon^{-2})} $ & \thead{lower bound:$\nabla \times u_{0} =0$; \\ upper bound:$\gamma=2$} \\
	\hline  Rammaha\cite{Rammaha} & 2D, formation of singularity ($ T_{\epsilon} < \infty $) & $T_{\epsilon}\le C\epsilon^{-2}$ for $\gamma=2$  \\
	\hline  Alinhac\cite{Alinhac} & 2D, $ \lim\limits_{\epsilon \rightarrow \infty} \epsilon^{2} T_{\epsilon} = C $ &  rotationally invariant data \\
	\hline  $^{**}$Jin and Zhou\cite{JINZHOU} & $ T_{\epsilon} \leq C \epsilon^{-n} (n=1,2), \exp{(C\epsilon^{-1})} (n=3) $ & for $\gamma=2$ \\
	\hline Lai and  Schiavone \cite{LS23} & $ T_{\epsilon} \leq C \epsilon^{-n} (n=1,2), \exp{(C\epsilon^{-1})} (n=3) $ & for $\gamma>1$ \\
	\hline
	\end{tabular}
\end{table}

\begin{rem}
\textbf{(1)} ($^{*}$): In particular, if the vorticity vanishes, Sideris extended the generic lower bound
\begin{equation} 
	T_{\epsilon} \geq C \epsilon^{-1},
\end{equation}
which typically holds for symmetric hyperbolic systems in all space dimensions
(see \cite{KATO}\cite{Majda} for the general theory).

\textbf{(2)} ($^{**}$): The upper bound of the lifespan estimate in $3$-$D$ improves the one in ($^{*}$), which in turn shows the optimality in $3$-$D$ when $\gamma=2$. 

On the other hand, putting together the above results, we find that the lower bounds corresponding to ($^{**}$) hold for any $\gamma>1$ and vanishing vorticity data, namely
\begin{equation} 
	\begin{aligned}
		T_{\epsilon} \geq
		\begin{split}
			\left\lbrace
			\begin{array}{lr}
				C \epsilon^{ -1 }, &  n=1,\\
				C \epsilon^{ -2 }, &  n=2,\\
				\exp{(C\epsilon^{-1})}, & n=3.
			\end{array}	
			\right. 
		\end{split}	
	\end{aligned}
\end{equation}

\end{rem}

\subsubsection{Euler system with damping}
Let us turn our attention to the corresponding problem in the presence of a damping term. Wang and Yang in \cite{WY01} proved global existence in $\R^{n} (n\ge 1)$ for the system $ ( \ref{1.2} ) $ with $\mu > 0$, $\lambda = 0$ and small initial perturbation of some constant state -- see also a similar result by Sideris, Thomases and Wang \cite{TCSBTDW} in $\R^{3}$.

If $\lambda>0$, then system \eqref{1.2} is called compressible Euler system with time-dependent damping. It attracts much attention and there is extensive literature on the long time behavior for this problem, see \cite{CYZZ18,CYZ23,CLLMZ20,GLM20,GHW21,GHJW23,HP03,JM230,JM23,SY21,SY22,SWY23} and references therein. In particular, Hou, Witt and Yin in \cite{HOUIWYIN} and Hou and Yin in \cite{HOUYIN} studied the system $ ( \ref{1.2} ) $ in $2$-$D$ and $3$-$D$, proving:
\begin{table}[htb]
	\centering
	\caption{Euler system with time-dependent damping in $\R^2$ and $\R^3$ (\cite{HOUIWYIN, HOUYIN})}
	\label{table:2}
	\begin{tabular}{|c|c|}
		\hline  \textbf{Condition} & \textbf{Result}  \\
		\hline  $0 \leq \lambda <1$, $\mu > 0$ & global existence  \\
		\hline  $\lambda =1 $, $\mu > 3-n$, $\nabla\times u_0$ & global existence \\
		\hline  $\lambda >1 $, $\mu > 0$ & finite time blow-up  \\
		\hline   $\lambda =1 $, $0 < \mu \leq 1$, $n=2$ & finite time blow-up  \\
		\hline
	\end{tabular}
\end{table}

Also, they established the upper bound of the lifespan estimate
\begin{equation} \label{1.7}
	T_{\epsilon} \leq \exp{(C \epsilon^{-2})}
\end{equation}
for $\gamma=2$.

In \cite{Pan1} and \cite{Pan2}, Pan studied the corresponding Cauchy problem in $1$-$D$, showing that (Table 3):
\begin{table}[htb]
	\centering
	\caption{Euler system with time-dependent damping in $1$-$D$ (\cite{Pan1,Pan2})}
	\label{table:3}
	\begin{tabular}{|c|c|}
		\hline  \textbf{Condition} & \textbf{Result}  \\
		\hline  $0 \leq \lambda <1$, $\mu > 0$ & there exists a global solution \\
		\hline  $\lambda =1 $, $\mu > 2$ & there exists a global solution\\
		\hline  $\lambda =1 $, $0 \leq \mu \leq 2$ & the $C^{1}$-solutions will blow up in finite time  \\
		\hline   $\lambda >1 $, $\mu \geq 0$ & the $C^{1}$-solutions will blow up in finite time  \\
		\hline
	\end{tabular}
\end{table}

Moreover, in the latter case, the same lifespan estimate as $ ( \ref{1.7} ) $ was established in \cite{Pan1} for $\gamma=2$. From these results, we may conclude that the point $(\lambda , \mu) = (1, 3-n)$ is the threshold for global existence and finite time blow up for the Cauchy problem \eqref{1.2} in $\R^n (n=1, 2, 3)$. And this is the reason we call "critical" time-dependent damping in the title. We should mention that the critical time-dependent damping admits significant physical background. Actually, nonlinear wave equations with critical damping for $\theta$ and $u$ (see \eqref{wave eqn of theta} and \eqref{wave eqn of u}) will be derived in section $2$, which are closely related to the generalized nonlinear Tricomi equations, appearing in gas dynamic problems, connected to gas flows with nearly sonic speed, see \cite{HWY17,LS22} and references therein. For more introduction to the nonlinear wave equations with critical damping, we refer to \cite{LTW17,IS18,LST20} and references therein.

In \cite{Sugiyama}, Sugiyama studied the system $ ( \ref{1.2} ) $ in $\R$ in Lagrangian coordinate, which is the so-called p-system (with damping)
\begin{equation} 
	\begin{aligned}
		\begin{split}
			\left\lbrace
			\begin{array}{lr}
				v_{t} - u_{x} = 0,\\
				u_{t} + p_{x} + a(t,x) v= 0,\\
				u(0,x) = \epsilon u_{0}(x), v(0,x) = 1+ \epsilon v_{0}(x).
			\end{array}	
			\right. 
		\end{split}	
	\end{aligned}
\end{equation}
where $ p \equiv p(v) = \frac{v^{- \gamma}}{\gamma} $ and $v = \frac{1}{\rho}$ is the specific volume. For the space-independent damping $a(t,x) = \frac{\mu}{(1+t)^{\lambda}}$, the sharp lifespan estimate
\begin{equation} \label{1dlf}
	\begin{aligned}
		T_{\epsilon} \thicksim
		\begin{split}
			\left\lbrace
			\begin{array}{lr}
				C \epsilon^{ -1 }, &  \lambda > 1 \text{ and } \mu \geq 0,\\
				C \epsilon^{ -\frac{2}{2-\mu} }, &  \lambda = 1 \text{ and } 0 \leq \mu <2,\\
				\exp{(C\epsilon^{-1})}, & \lambda = 1 \text{ and } \mu =2
			\end{array}	
			\right. 
		\end{split}	
	\end{aligned}
\end{equation}
was established.

From the sharp lifespan estimate \eqref{1dlf} in $1$-$D$, it is reasonable to believe 
the lifespan should depend on the constant $\mu$ in the blow up case, if the decay rate of the damping coefficient is critical ($\lambda=1$). Actually, the second author and Schiavone \cite{LS23} introduced the Orlicz spaces techniques
and obtained the following upper bound of lifespan estimate for \eqref{1.2} with $\lambda=1, \gamma>1$ and $n=1, 2$
\begin{equation}\label{eq:thm2}
	T_\e \le
	\left\{
	\begin{aligned}
		&C\e^{-\frac{2}{3-n-\mu}}
		&&\text{if $\mu<3-n$,}
		\\
		&\exp\left(C \e^{-1}\right)
		&&\text{if $\mu=3-n$,}
	\end{aligned}
	\right.
\end{equation}
where the case of $n=1$ coincide with the result in \eqref{1dlf}. It is natural to verify the optimality of the lifespan \eqref{eq:thm2} when $n=2$, and this is the goal of the present paper. The sharpness has been proved for $n=1$ in \cite{Sugiyama} by employing the characteristic method, which seems not traceable in high dimensional case. To our end, we will use vector field method instead. We will confront two main difficulties: the first one comes from the linear damping term, which will be overcome by using $\widehat{Z}$ instead of $Z$, to get better decay for the linear error term. The second one is due to the quasilinear nonlinear terms, which will be handled by employing the symmetric structure and the equations. The strategy of our proof can be sketched in Flow chart $1$ below.

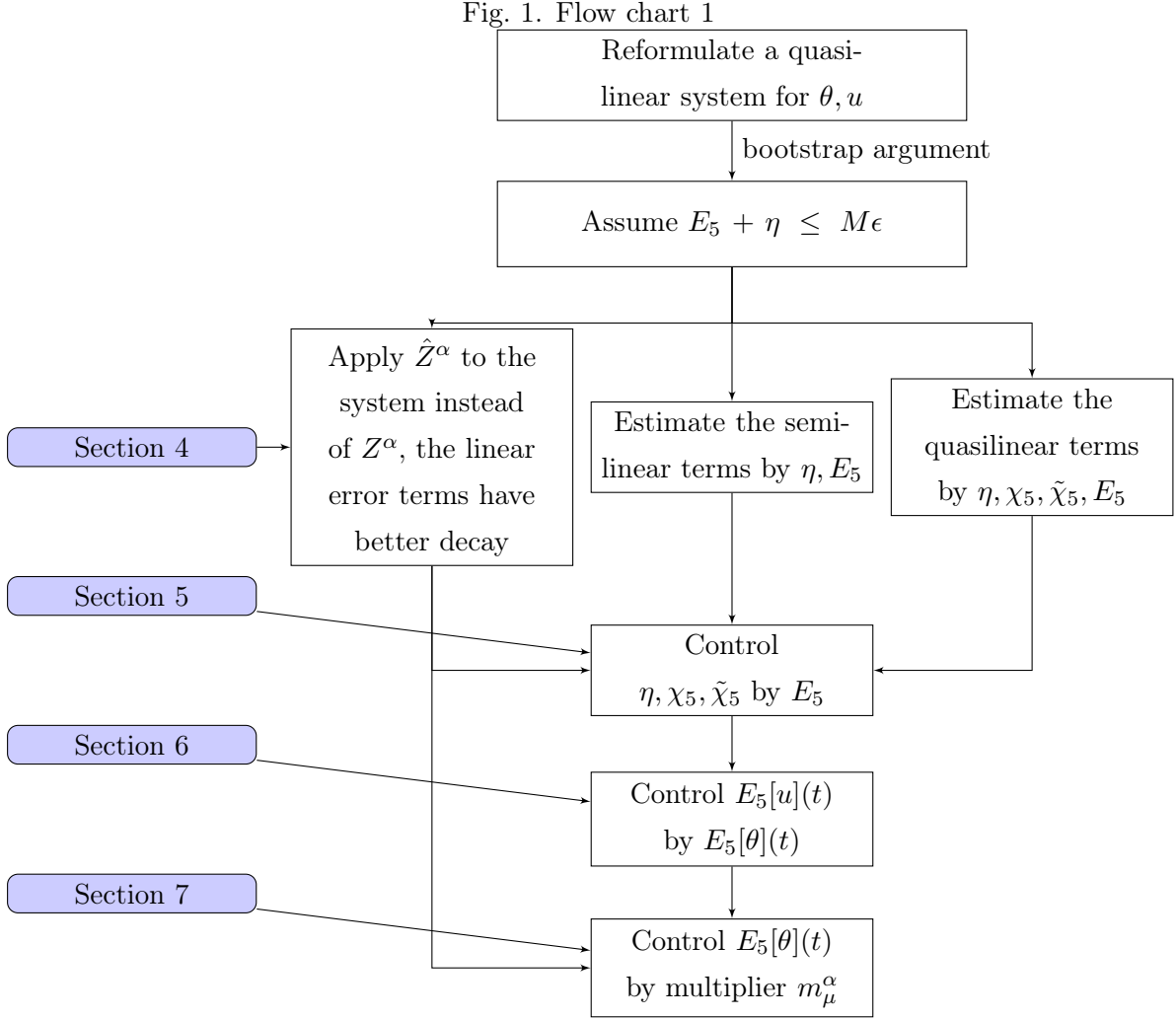
\begin{figure}[!ht]
	\centering
\caption{Flow chart 1}
	%\hspace*{-1pt}
	\begin{tikzpicture}[ >=latex', auto]
		% nodes
		
		\node [intg] (Reformulate) { Reformulate a quasilinear system for $ \theta,u $};
		\node [intg, below of = Reformulate, node distance = 2cm] (Assume) { Assume $ E_{5}+ \eta \leq M \epsilon$};
		\node [int,  below of = Assume, node distance = 3cm] (semilinear) {Estimate the semilinear terms by $\eta, E_{5} $};
		\node [int,  left of = semilinear, node distance = 4cm] (Zalpha) { Apply $\hat{Z}^{\alpha}$ to the system instead of $Z^{\alpha}$, the linear error terms have better decay};
		\node [block, left of = Zalpha, node distance = 4cm] (Section4) { Section 4};
		\node [int, right of = semilinear, node distance = 4cm] (quasilinear) {Estimate the quasilinear terms by $\eta , \chi_{5}, \tilde{\chi}_{5}, E_{5} $};
		\node [int, below of = semilinear, node distance = 3cm] (auxiliary) {Control $\eta , \chi_{5}, \tilde{\chi}_{5}$ by $E_{5}$};
		\node [block, below of = Section4, node distance = 2cm] (Section5) { Section 5};
		\node [int, below of = auxiliary, node distance = 2cm] (Eu) {Control $E_{5}[u](t)$ by $E_{5}[\theta](t)$};
		\node [block, below of = Section5, node distance = 2cm] (Section6) { Section 6};
		\node [int, below of = Eu, node distance = 2cm] (Etheta) {Control $E_{5}[\theta](t)$ by multiplier $ m_{\mu}^{\alpha} $};
		\node [block, below of = Section6, node distance = 2cm] (Section7) { Section 7};
		
		% edges
		
		\path [line] (Reformulate) -- node {bootstrap argument}(Assume);
		\path [line] (Assume) -- ($(Assume.south)+(0,-0.75)$) -| (quasilinear) node[above,pos=0.25] {} ; 
		\path [line] (Assume) -- ($(Assume.south)+(0,-0.75)$) -- (semilinear) node[above,pos=0.25] {};
		\path [line] (Assume) -- ($(Assume.south)+(0,-0.75)$) -| (Zalpha) node[above,pos=0.25] {} ; 
		\path [line] (Section4) -- (Zalpha);
		\path [line] (Zalpha) |- (auxiliary);
		\path [line] (semilinear) -- (auxiliary);
		\path [line] (quasilinear) |- (auxiliary);
		\path [line] (Section5) -- (auxiliary);
		\path [line] (auxiliary) -- (Eu);
		\path [line] (Eu) -- (Etheta);
		\path [line] (Section6) -- (Eu);
		\path [line] (Section7) -- (Etheta);
		\path [line] (Zalpha) |- (Etheta);
		
		%\path [line] (blood) |- ([xshift = -2.3cm, yshift = 1cm]metabolomics.north west) -- (genetics);
	\end{tikzpicture}
\end{figure}

Our main result is precisely stated as follows. 

\begin{thm} \label{thm1}
	Consider the system \eqref{Euler eqns}\eqref{pressure law} with $0 < \mu \leq 1 $. If the perturbations $(\rho_0,u_0)$ satisfy
	\begin{itemize}
		\item $(\rho_0,u_0)\in C_c^\infty(B_R(0))$, where $B_R(0)=\{x\in\R^2:|x|<R\}$;
		\item $ \operatorname{curl} u_{0}(0,x) = 0$ holds for any $x\in\mathbb{R}^2$;
	\end{itemize}
	then there exist constants $0<\epsilon_0\ll1$ and $C>0$, such that for all $ \epsilon\in(0,\epsilon_0] $, \eqref{Euler eqns}\eqref{pressure law} admits a classical solution $ (\rho,u) \in C^\infty ([ 0,T_{\epsilon} ]\times\R^2)$, where
	\begin{equation} 
		\begin{aligned}
			T_{\epsilon} =
			\begin{split}
				\left\lbrace
				\begin{array}{lr}
					C \epsilon^{ - \frac{2}{1-\mu} }, \ \ 0 < \mu < 1 \\
					C \exp ( \epsilon^{ - 1 } ) , \ \ \mu = 1.
				\end{array}	
				\right. 
			\end{split}	
		\end{aligned}
	\end{equation}
	The quantities $\epsilon_0,C$ only depend on the perturbations $(u_0,\rho_0)$ and the constants $\mu,\gamma$. Moreover, we have $ {\rho} >0 $ on $[ 0,T_{\epsilon} ]\times\R^2$. 
	
\end{thm}

\begin{rem}
	The initial data are small perturbations of constant states $ (\bar{\rho},0)$ with the background density $\bar\rho>0$. The "smallness" is quantified by
	the parameter $ \epsilon>0 $. Without loss of generality, we may assume $\bar{\rho}= 1$ and $\operatorname{supp}(\rho_0,u_0)\subset \overline{ B_{1/2}(0)}$. 
\end{rem}

\begin{rem}
	Theorem \ref{thm1} provides a lower bound $T_{max}> T_\epsilon$ for the lifespan $T_{max}$ of the solution, which is defined as the largest time $T$ such that $(\rho,u)\in C^\infty([0,T)\times\R^2)$ and $\rho>0$ on $[0,T)\times\R^2$. 
\end{rem}

\begin{rem}
	We define the local sound speed by $c=\sqrt{\partial p/\partial\rho }=\rho^{\frac{\gamma-1}{2}}$, and let $\bar c=\bar \rho^{\frac{\gamma-1}{2}}=1$ be the sound speed of the background solution. We will see that the support of $(\rho-1,u)$ evolves at a finite speed not exceeding $\bar c=1$. More precisely, we have
	\begin{equation}
		\operatorname{supp}(\rho-1,u)\subset \overline{B_{1/2+t}(0)}.
	\end{equation}
	The proof will be sketched in Lemma \ref{finite speed of propagation}. 
\end{rem}

The paper will be organized as follows. In section $2$ we reformulate the original system to a hyperbolic system, and then the local well-posedness, blow-up criterion and finite speed of propagation property are introduced. Also, the nonlinear wave equations with time-dependent damping for the local sound speed $\theta$ and velocity field $u$ are derived. Then some useful properties of the vector fields and inequalities are introduced. In section $3$ we reduce the proof to a bootstrap argument, and hence some necessary energy functionals are introduced. The linear error terms from the damping term, the semilinear terms and the quasilinear terms are estimated by the energy functionals respectively in section $4$. In section $5$, we show the auxiliary energies $\eta, \chi, \widetilde{\chi}$ are controlled by $E[\theta, u]$, while $E_5[u]$ is proved to be controlled by $E_5[\theta]$ in section $6$. In section $7$ we demonstrate the estimate for $E_5[\theta]$ by using a multiplier $m^\alpha_\mu$ and finally close the bootstrap argument. What is more, some useful lemmas are listed at the end of the paper, i.e. right before the references.

\section{Preliminaries}
In this section, we reformulate the damped Euler equations to a hyperbolic system by subtracting the background stationary state. Moreover, in order to apply the techniques from nonlinear wave equations, we deduce wave equations with time-dependent damping for both the local sound speed and the velocity field. 
\subsection{Reformulating as a hyperbolic system}
Let us first reformulate system $ (\ref{Euler eqns}) $ to a hyperbolic system. Define 
\begin{equation} 
	\begin{aligned}
		\theta \triangleq \frac{1}{\gamma -1} (\rho^{\gamma -1} -1 ) = \frac{1}{\gamma -1} (c^{2}(\rho) -1),
	\end{aligned}
\end{equation}
where $c=\sqrt{p^{\prime}(\rho) }=\rho^{\frac{\gamma-1}{2}}$ is the local sound speed. As mentioned above, we assume that $\bar{c} = c( \bar{\rho}) =1 $. The problem $ (\ref{Euler eqns}) $ can be rewritten as 
\begin{equation}\label{2}
	\left\{
	\begin{aligned}
		&\partial_t\theta+u\cdot\nabla\theta+\left(1+(\gamma-1)
		\theta\right)\nabla\cdot u=0,~~~~~(t, x)\in [0,T] \times \R^{2} ,
		\\
		&\partial_tu+ \frac{\mu}{1+t} u + u\cdot \nabla u+\nabla\theta=0,~~~~~(t, x)\in [0,T] \times \R^{2} ,\\
		&\theta(0,x)= \frac{1}{\gamma-1} [ ( 1 + \epsilon \rho_{0}(x) )^{ \gamma-1 } -1 ] \triangleq \epsilon \theta_{0}(x) + \epsilon^{2} g (x, \epsilon) ,\\
		& u(0,x)=\epsilon u_{0}(x),~~~x\in \R^{2}.
	\end{aligned}
	\right.
\end{equation}
The initial data of $\theta$ can be obtained by using Taylor expansion with integral remainder: 
\begin{equation} 
	\label{initial data of theta}
	\begin{aligned}
		& \theta_{0}(x) = \rho_{0}(x),\ g (x, \epsilon) = (\gamma-2)\rho_{0}^{2}(x)\int_{0}^{1}  ( 1 + \epsilon \sigma \rho_{0}(x) )^{ \gamma-3 }  ( 1 - \sigma) d \sigma.
	\end{aligned}
\end{equation}
Note that $ g (x, \epsilon)$ is smooth in $(x , \epsilon)$ and has compact support in $x$:
\begin{equation} 
	\begin{aligned}
		& g (x, \epsilon) \in C^{\infty} (\R^{2} \times [0,1]),\  \operatorname{supp} g(x, \epsilon) \subset \overline{B_{1/2}(0)} \times [0,1].
		\label{properties of g} 
	\end{aligned}
\end{equation}

\begin{rem}
	\label{vorticity free}
	Let $ \omega = \operatorname{curl} u = \partial_{1} u_{2} - \partial_{2} u_{1} $ be the vorticity. By applying $\operatorname{curl}$ to both sides of the second equation in \eqref{2}, one has
	\begin{equation} \label{equ of vorticity}
		\begin{aligned}
			\partial_t \omega + \frac{\mu}{1+t} \omega + u\cdot \nabla \omega + \omega \nabla \cdot u =0.
		\end{aligned}
	\end{equation}
	Thus $\omega$ satisfies a transport-type equation, with $-\left(\frac{\mu}{1+t}+\nabla\cdot u\right)\omega$ as its forcing term. This forcing term vanishes when $\omega=0$. Consequently, if $ \operatorname{curl} u_{0} = 0  $, then $ \operatorname{curl} u \equiv 0  $ holds as long as the smooth solution $( \theta,u )$ of \eqref{2} exists. 
\end{rem}

\subsubsection{Local theory for the reformulated system}
By applying Theorem 2.1 and Theorem 2.2 in \cite{Majda} to \eqref{Euler eqns}, one can obtain the local well-posedness and blow-up criterion for \eqref{Euler eqns}. If we translate this result to the system  \eqref{2} of $(u,\theta)$, we will arrive at the following theorems.  

\begin{thm} \label{LWP for theta,u}
	\textbf{[Local well-posedness]} Let $G=\mathbb{R}^2\times(-\frac{1}{\gamma-1},+\infty)$ and $G_1\subset\subset G$ be two open domains. Suppose $k \geq 3$ and the initial data of \eqref{2} satisfy
	\begin{itemize}
		\item $(u(0,x),\theta(0,x)) \in H^{k}_x(\R^2)$;
		\item $ (u(0,x),\theta(0,x))\in G_{1} $ for any $x\in\R^2$.
	\end{itemize}
	Then we have that
	\begin{itemize}
		\item there exists a time $T>0$ depending only on $G_1$ and the $H^k_x$ norm of $(u(0,x),\theta(0,x))$;
		\item there exists a open domain $G_2\subset\subset G$;
		\item there exists a vector-valued function $(u,\theta)\in C^1([0,T]\times\R^2)\cap C([0,T];H^k(\R^2)) \cap C^{1}([0,T];H^{k-1}(\R^2))$ , solving \eqref{Euler eqns} and having $(u(0,x),\theta(0,x))$ as its initial data. Furthermore, we have that $(u(t,x),\theta(t,x))\in G_2$ for any $(t,x)\in[0,T]\times\R^2$. 
	\end{itemize}
\end{thm}

\begin{thm} \label{thmmajda2}
	\textbf{[Blow-up criterion]} Suppose $k \geq 3$ and $(u,\theta)\in C^1([0,T_{max})\times\R^2)\cap C([0,T_{max});H^k(\R^2)) \cap C^{1}([0,T_{max});H^{k-1}(\R^2))$ solves \eqref{Euler eqns} with $T_{max}$ to be its maximal lifespan. Then $T_{max}<\infty$ if and only if at least one of the following scenarios occur:
	\begin{itemize}
		\item Formation of shock waves or ODE type blow-up:
		\begin{equation}
			\limsup\limits_{t\uparrow T_{max}}\left(\|u(t,\cdot)\|_{C^1(\R^2)}+\|\theta(t,\cdot)\|_{C^1(\R^2)} + \|\partial_{t} u(t,\cdot)\|_{L^{\infty}(\R^2)}+\|\partial_{t}\theta(t,\cdot)\|_{L^{\infty}(\R^2)} \right)=\infty;
		\end{equation}
		\item Appearance of vacuum: there exists a sequence $\{(t_j,x_j)\}_{j=1}^\infty\subset[0,T_{max})\times\R^2$ such that $\lim\limits_{j\rightarrow\infty}t_j=T_{max}$ and
		\begin{equation}
			\lim\limits_{j\rightarrow\infty}\theta(t_j,x_j)=-\frac{1}{\gamma-1}.
		\end{equation}
	\end{itemize}
\end{thm}

\begin{cor}
	If $ (u(0,x),\theta(0,x))\in C^{\infty}_{c} (\R^2) \subset \bigcap\limits_{ k \geq 1} H^{k} (\R^2) $, then  $ (u,\theta)\in \bigcap\limits_{ k \geq 1} C ([0,T];H^{k}(\R^2)) $ for $T\in(0,T_{max})$, where $T_{max}$ is the lifespan. Therefore, $ (u,\theta) (t) \in \bigcap\limits_{ k \geq 1} H^{k}(\R^2) \subset C^{\infty} (\R^2) $ for all $t \in [0,T]$. Combining this with the following finite speed of propagation, we have that $ (u,\theta) (t) \in C^{\infty}_{c} (\R^2) $ for all $t \in [0,T]$. The blow-up criterion for the smooth solutions is the same as that for the $H^k$ solutions.  
\end{cor}

\subsubsection{Evolution of the support}
Using \eqref{2}, we can prove that the support of $(u,\theta)$ evolves at a finite speed not exceeding 1. This finite speed of propagation property is a classical result and can be found in various textbooks on hyperbolic systems, we refer to \cite{Sideris1985,Sideris1991,TCSBTDW} for a more exhaustive discussion. However, for the readers' convenience, we provide a brief outline of the proof here.
\begin{lem}
	\label{finite speed of propagation}
	Suppose $(u,\theta)\in C^\infty([0,T]\times\R^2)$ solves system \eqref{2} with the initial data $(u(0,x),\theta(0,x))$ supported in $\overline{B_{1/2}(0)}$, we have that
	\begin{equation}
		\operatorname{supp}(u(t,\cdot),\theta(t,\cdot))\subset \overline{B_{1/2+t}(0)},\ \ \forall t\in[0,T].
		\label{estimate on the support} 
	\end{equation}
	The expanding speed 1 comes from the fact that the local sound speed $c=1$ outside the support. 
\end{lem}
\begin{proof}
	Define $f(t,K)=\|(u,\theta)\|_{L^2(B_{1/2+Kt}(0)^c)}$, and $g(t,K)=\|u\|_{L^\infty(B_{1/2+Kt}(0)^c)}$. Then $f,g$ are continuous functions on $[0,T]\times[1,\infty)$ since $(u,\theta)\in C^\infty([0,T]\times\R^2) \cap C( [0,T]; H^{3} (\R^2) ) $, and $f,g$ are decreasing with respect to $K$. Multiplying $\theta$ to the first equation of \eqref{2}, taking the inner product of $u$ and the second equation of \eqref{2}, summing up the equations and integrating over the region $B_{1/2+Kt}(0)^c$, We obtain for $t\in[0,T]$
	\begin{equation}
		\partial_tf^2(t,K)+\left[K-1-g(t,K)\right]\|(\rho,u)\|_{L^2(\partial B_{1+Kt}(0))}^2\le C(\|(u,\theta)\|_{C^1})f^2(t,K).
		\label{auxi ineq for finite speed of propagation}
	\end{equation}

	Now fix $K>1$. We impose the bootstrap assumption $g(t,K)\le K-1$. Clearly $f(0,K)=g(0,K)=0$, thus the bootstrap assumption holds true at least for a short time. Suppose $g(t,K)\le K-1$ remains true on $[0,T_1]$ with $T_1\le T$. Then \eqref{auxi ineq for finite speed of propagation} together with the Gronwall inequality imply $0\le f^2(t,K)\le f^2(0,K)e^{C(\|(u,\theta)\|_{C^1})t}=0$ on $[0,T_1]$. Thus $g(t,K)=0$ on $[0,T_1]$. By the standard bootstrap argument, we know that $g(t,K)=f(t,K)=0$ for $t\in[0,T]$. Since $K>1$ is arbitrary, by the continuity of $f,g$, we have $f(t,1)=g(t,1)=0$ for $t\in[0,T]$. This proves the lemma. 
	
\end{proof}

\subsection{Reformulating as wave equations}

Let us first derive the damped wave equations of $\theta$ in \eqref{2} in detail. It follows from the first equation in \eqref{2} that
\begin{equation} \label{no1}
	\partial_{t} \operatorname{div} u = - \frac{1}{( 1+ ( \gamma -1 )\theta) )}\left ( \partial_{t}^{2} \theta + u \cdot \nabla \partial_{t} \theta + \partial_{t} u \cdot \nabla \theta  + ( \gamma -1 ) \partial_{t} \theta \operatorname{div} u\right).
\end{equation}
Taking divergence on the second equation in \eqref{2} gives
\begin{equation} \label{no2}
	\operatorname{div} \partial_{t} u + \frac{\mu}{1+t} \operatorname{div} u + \Delta \theta + u \cdot \nabla \operatorname{div} u + \sum\limits_{i,j=1}^{2} \partial_{i} u_{j} \partial_{j} u_{i} =0,
\end{equation}
where $ \Delta=\partial_{1}^{2} +\partial_{2}^{2} $. Substituting \eqref{no1} into \eqref{no2} yields the damped wave equation for $\theta$:

\begin{equation} \label{wave eqn of theta}
	\boxed{\Box \theta + \frac{\mu}{1+t} \partial_{t}  \theta  = F_{\theta}},
\end{equation}
where $\Box = \partial_{t}^{2} - \partial_{x_{1}}^{2}- \partial_{x_{2}}^{2}$ is the two dimensional wave operator, and the forcing terms are
\begin{equation}
	\begin{aligned}
		F_{\theta}  \triangleq &F_{\theta,1}  + F_{\theta,2} ,\\
		F_{\theta,1} \triangleq&  - \frac{\mu}{1+t} u \cdot \nabla \theta - \sum\limits_{i,j=1}^{2} u_{i} \partial_{i} u_{j} \partial_{j} \theta - \partial_{t} u \cdot \nabla \theta \\
		& + (1 + ( \gamma -1 )\theta) \left[\sum\limits_{i,j=1}^{2} \partial_{i} u_{j} \partial_{j} u_{i} +  ( \gamma -1 ) |\nabla \cdot u|^{2} \right], \\
		F_{\theta,2} \triangleq& ( \gamma -1 ) \theta \Delta \theta - 2u \cdot \nabla \partial_{t} \theta - \sum\limits_{i,j=1,2} u_{i} u_{j} \partial_{ij}^{2} \theta.
	\end{aligned}
\end{equation}
Noting that $F_{\theta,1}$ contains first order derivative terms, and it has the form $ \sum\limits_{0 \leq \alpha, \beta \leq 2} \sum\limits_{i=1}^{2} g_{\alpha, \beta, i} (\theta , u) \partial_{\alpha}\theta \partial_{\beta} u_{i} $ with $ g_{\alpha, \beta, i} (\theta , u) $ being linear with respect to $(\theta , u)$. $F_{\theta,2}$ consists of the second order derivative terms in $F_{\theta}$.

Next, we derive the damped wave equation for $u$. Differentiating the first equation in \eqref{2} and using the fact that $\operatorname{curl}u=0$ give
\begin{equation} \label{uno1}
	\nabla \partial_t \theta+ \Delta u + \nabla \left(u\cdot\nabla\theta+ (\gamma-1) \theta \nabla\cdot u \right)=0.
\end{equation}
It follows from the second equation in \eqref{2} that
\begin{equation} \label{uno2}
	\partial_{tt}u+ \partial_t \left(\frac{\mu}{1+t} u\right) + \partial_t (u\cdot \nabla u)+\nabla  \partial_t \theta=0
\end{equation}
Substituting \eqref{uno1} into \eqref{uno2} yields the damped wave equation for $u$:
\begin{equation} \label{wave eqn of u}
	\boxed{\Box u + \frac{\mu}{1+t} u_{t} - \frac{\mu}{(1+t)^{2}} u = F_{u}},
\end{equation}
where
\begin{equation}
	\begin{aligned}
		F_{u}=& - \partial_t (u\cdot \nabla u) + \nabla \left(u\cdot\nabla\theta+ (\gamma-1) \theta \nabla\cdot u \right)\\
		\triangleq& F_{u,1} + F_{u,2} ,\\
		F_{u,1} \triangleq& - (\partial_{t} u \cdot \nabla) u + \nabla u \cdot \nabla \theta  + ( \gamma -1 ) \nabla \theta \nabla\cdot u, \\
		F_{u,2} \triangleq&  - (u \cdot \nabla )\partial_{t} u + u \cdot \nabla^2 \theta + ( \gamma -1 ) \theta \Delta u .
	\end{aligned}
\end{equation}
Equivalently, we can write the above equation in component form:
\begin{equation}
	\begin{aligned}
		\Box u_i+\partial_t\left(\frac{\mu}{1+t}u_i\right)=&\underbrace{-\partial_tu_j\partial_ju_i+\partial_iu_j\partial_j\theta+(\gamma-1)\partial_i\theta\partial_ju_j}_{F_{u_i,1}}\\
		&\underbrace{-u_j\partial_j\partial_tu_i+u_j\partial_{ij}^2\theta
+(\gamma-1)\theta\partial_j\partial_ju_i}_{F_{u_i,2}},
	\end{aligned}
\end{equation}
where the Einstein summation notation is used. 

In the following we focus on studying the wave equations with time-dependent damping \eqref{wave eqn of theta} and \eqref{wave eqn of u} of ($\theta,u$).

\subsubsection{Vector fields and commutators}
We define several vector fields and deduce some commutation relations in this subsection. Denote
\begin{equation}
	\begin{aligned}
		&\partial = ( -\partial_{t}, \partial_{x_{1}}, \partial_{x_{2}} ),\ S = t \partial_{t} + x_{1}\partial_{x_{1}} + x_{2}\partial_{x_{2}} ,\ \Omega = x_{1}\partial_{x_{2}} -  x_{2}\partial_{x_{1}}.
	\end{aligned}
\end{equation}
It is well-known for the communication relations 
\begin{equation}
	[ \partial, \Box ] = 0 ,\ \ [ \Omega, \Box ] = 0 ,\ \ [S,\Box]=-2\Box,
\end{equation}
where $[A, B] = AB-BA$ stands for the commutator between $A$ and $B$.

To further utilize these vector fields, we define
\begin{equation}
	Z=(Z_0,Z_1,Z_2,Z_3,Z_4)=(\partial,S-1,\Omega),
	\label{def of Z}
\end{equation}
\begin{equation}
	\hat Z=(\hat Z_0,\hat Z_1,\hat Z_2,\hat Z_3,\hat Z_4)=(\partial,S+1,\Omega).
	\label{def of hatZ}
\end{equation}
For any multi-index $\alpha=(\alpha_0,\cdots,\alpha_4)\in\mathbb{N}^5$, we define $Z^\alpha=Z_0^{\alpha_0}\cdots Z_4^{\alpha_4}$ and $Z^{(0,0,0,0,0)}=1$. Similarly, define $\hat Z^\alpha=\hat Z_0^{\alpha_0}\cdots \hat Z_4^{\alpha_4}$. 

Direct calculation yields 
\begin{lem}\label{commutator}
	We have that:
	\begin{equation}
		[\partial,S-1]=\partial,\ \ [\partial_{x_i},\Omega]=\varepsilon_{ij}\partial_{x_j},\ \ [S,\Omega]=0,
	\end{equation}
	where $i\in\{1,2\}$, $\partial$ represents one of the operator in $\{\partial_t,\partial_{x_1},\partial_{x_2}\}$, and
	\begin{equation}
		\varepsilon_{ij}=\left\{
		\begin{aligned}
			&1,\ \ &(i,j)=(1,2),\\
			&-1,\ \ &(i,j)=(2,1),\\
			&0,\ \ &i=j.\\
		\end{aligned}\right.
	\end{equation}
\end{lem}

The above lemma shows that $\mathfrak{g}=\{\sum_i\lambda_iZ_i\}$ is a Lie algebra over $\R$. Its universal enveloping algebra is denoted by $U(\mathfrak{g})$. Besides, the $\R$-algebra generated by $\{Z^\alpha:\alpha_i\ge0,i=0,\ldots,4\}$ is denoted by $\mathcal{A}$. Using the universal property of $U(\mathfrak{g})$, there exists a unique algebra homomorphism $h:U(\mathfrak{g})\rightarrow\mathcal{A}$. Clearly $h$ is surjective. Consequenty, it sends a basis of $U(\mathfrak{g})$ to a generating set of $\mathcal{A}$. According to Poincar\'e-Birkhoff-Witt theorem, $\{Z^\gamma|\gamma_i\ge0\}$ forms a basis of $U(\mathfrak{g})$. Thus, $\{Z^\gamma|\gamma_i\ge0\}$ generates $\mathcal{A}$. More precisely, any element in $\mathcal{A}$ can be written in the form
\begin{equation}
	\sum_{\text{finite sum}} C_\gamma Z^\gamma.
\end{equation}
Since $U(\mathfrak{g})$ has a canonical filtration, we know that
\begin{equation}
	Z^\alpha Z^\beta=\sum_{|\gamma|\le|\alpha|+|\beta|} C_{\alpha,\beta,\gamma} Z^\gamma.
	\label{cononical filtration}
\end{equation}

What is more, we may prove by using mathematical induction
\begin{lem}[\cite{LIZHOUbook}, Corollary 3.1.1]For any multi-index $\alpha,\beta>0$, we have the representation of the commutator:
\begin{equation}
	[Z^\alpha,Z^\beta]=\sum_{\substack{|\gamma|\le|\alpha|+|\beta|-1\\ \gamma_0+\gamma_1+\gamma_2\ge1}}C_{\alpha,\beta,\gamma}Z^\gamma.
\end{equation}
Notably any monomial $Z^\gamma$ that appears in the right hand side satisfies the condition $\gamma_0+\gamma_1+\gamma_2\ge1$. We abbreviate this fact as
\begin{equation}
	[Z^\alpha,Z^\beta]\overset{\text{abbr.}}{=}\sum_{|\gamma|\le|\alpha|+|\beta|-2}C_{\alpha,\beta,\gamma}\partial Z^\gamma
	\overset{\text{abbr.}}{=}\partial Z^{<|\alpha|+|\beta|-1}.
	\label{representation of commutators}
\end{equation}
\end{lem}

\begin{rem}
	Hereafter, we will always use the notation
	\begin{equation}
		\label{abbreviation}
		Z^{<k}=\sum_{|\beta|<k}C_\beta Z^\beta,\ Z^{\le k}=\sum_{|\beta|\le k}C_\beta Z^\beta,\ Z^{<\alpha}=\sum_{\beta<\alpha}C_\beta Z^\beta,\ Z^{\le \alpha}=\sum_{\beta\le \alpha}C_\beta Z^\beta,
	\end{equation}
	where $\beta\le\alpha$ means $\beta_i\le\alpha_i$ for all $i$. These notations may simplify the calculation, but make it hard to track the constants $C_\beta$. Therefore, these notations are only used in the case that \emph{all the constants $C_\beta$ are universal}. 
\end{rem}

\begin{cor} \label{lempartialZ}
	Suppose $k\ge0$ is an integer, then the following relation holds:
	\begin{equation} 
		\partial Z^{\le k}=Z^{\le k}\partial. 
		\label{equivalent form of DZ^alpha}
	\end{equation}
\end{cor}

\begin{rem}
	Suppose $\alpha\in\mathbb{N}^5$ is an multi-index, then we have
	\begin{equation}
		\label{relation between Z and Zhat}
		\hat Z^\alpha=Z^{\le|\alpha|},\ \ Z^\alpha=\hat Z^{\le|\alpha|}.
	\end{equation}
	By virtue of \eqref{abbreviation}, the above relations should be understood as
	\begin{equation}
		\hat Z^\alpha=\sum_{|\gamma|\le|\alpha|}C_{\alpha,\gamma}Z^\gamma,\ \ Z^\alpha=\sum_{|\gamma|\le|\alpha|}C_{\alpha,\gamma}\hat Z^\gamma.
	\end{equation}
\end{rem}

For further needs, we give some simple but important properties concerning the communication between $Z,\hat Z$ and the damping terms.

\begin{lem} \label{lemrelation}
	The following relations hold for any smooth function $f(t,x)$:
	\begin{equation}
		\hat Z_3 \Box f - \Box Z_3 f =0,
		\label{commute hat Z3 with box}
	\end{equation}
	\begin{equation}
		\hat Z_3\left(\frac{1}{1+t}\partial_tf\right)-\frac{1}{1+t}\partial_tZ_3f
=\frac{1}{(1+t)^2}\partial_tf,
	\end{equation}
	\begin{equation}
		\hat Z_3\left(\frac{1}{(1+t)^2}f\right)-\frac{1}{(1+t)^2}Z_3f=\frac{2}{(1+t)^3}f.
	\end{equation}
\end{lem}
Furthermore, we can prove the following identities by using induction argument.
\begin{cor}
	For $k\ge1$, we have 
	\begin{equation}
		\hat Z_3^k\left(\frac{1}{1+t}\partial_tf\right)-\frac{1}{1+t}\partial_tZ_3^kf=\sum_{j=0}^{k-1}\hat Z_3^{k-1-j}\left(\frac{1}{(1+t)^2}\partial_tZ_3^{j}f\right),
	\end{equation}
	\begin{equation}
		\hat Z_3^k\left(\frac{1}{(1+t)^2}f\right)-\frac{1}{(1+t)^2}Z_3^kf=\sum_{j=0}^{k-1}\hat Z_3^{k-1-j}\left(\frac{2}{(1+t)^3}Z_3^{j}f\right).
	\end{equation}
	For any multi-index $\alpha\in\mathbb{N}^5$, we have
	\begin{equation}
		\hat Z^\alpha\Box f-\Box Z^\alpha f=0,
		\label{commute hat Z with box}
	\end{equation}
	\begin{equation}
		\begin{aligned}
			\hat Z^\alpha\left(\frac{1}{1+t}\partial_tf\right)-\frac{1}{1+t}\partial_tZ^\alpha f=&\mathbbm{1}_{\alpha_0>0}\sum_{j=1}^{\alpha_0}\frac{\alpha_0!}{(\alpha_0-j)!}\frac{1}{(1+t)^{j+1}}\partial_tZ^{\alpha-je_0}f\\
			&+\mathbbm{1}_{\alpha_3>0}\sum_{j=0}^{\alpha_3-1}\hat Z^{\alpha-(j+1)e_3}\left(\frac{1}{(1+t)^2}\partial_tZ^{je_3}f\right),
		\end{aligned}
		\label{commute Z and damping, for theta}
	\end{equation}
	\begin{equation}
		\begin{aligned}
			\hat Z^\alpha\left(\frac{1}{(1+t)^2}f\right)-\frac{1}{(1+t)^2}Z^\alpha f=&\mathbbm{1}_{\alpha_0>0}\sum_{j=1}^{\alpha_0}\frac{\alpha_0!(j+1)}{(\alpha_0-j)!}\frac{1}{(1+t)^{j+2}}Z^{\alpha-je_0}f\\
			&+\mathbbm{1}_{\alpha_3>0}\sum_{j=0}^{\alpha_3-1}\hat Z^{\alpha-(j+1)e_3}\left(\frac{2}{(1+t)^3}Z^{je_3}f\right).
		\end{aligned}
		\label{commute Z and damping, for u}
	\end{equation}
    Here $\mathbbm{1}_{\alpha_0>0}$ is the characteristic function of the set $\{\alpha\in\mathbb{N}^5:\alpha_0>0\}$. $e_0=(1,0,0,0,0)$ and $e_3=(0,0,0,1,0)$ are vectors in $\mathbb{N}^5$. 
\end{cor}
\begin{proof}
	The identity \eqref{commute hat Z with box} follows directly from \eqref{commute hat Z3 with box}. The other two identites can be proved by using the formula $A^nB=\sum_{j=0}^n\binom{n}{j}[A^{(j)},B]A^{n-j}$, where $[A^{(j)},B]$ is defined iteratively by
	\begin{equation}
		\left\{
		\begin{aligned}
			&[A^{(j)},B]=\left[A,[A^{(j-1)},B]\right],\\
			&[A^{(0)},B]=B,
		\end{aligned}\right.
	\end{equation}
	and the facts that
	\begin{equation}
		\left[\partial_t^{(j)},\frac{1}{1+t}\right]=\frac{j!(-1)^j}{(1+t)^{j+1}},
	\end{equation}
	\begin{equation}
		\left[\partial_t^{(j)},\frac{1}{(1+t)^2}\right]=\frac{(j+1)!(-1)^j}{(1+t)^{j+2}}.
	\end{equation}
\end{proof}

\subsection{Equations of derivatives}
Applying $ \hat{Z}^{\alpha} $ to the equation \eqref{wave eqn of theta} of $ \theta $, we can obtain the equation of $Z^{\alpha} \theta$ by using \ref{commute hat Z with box}:
\begin{equation} \label{eqn of Z^alpha theta}
	\boxed{\Box Z^{\alpha} \theta + \frac{\mu}{1+t} \partial_{t} Z^{\alpha} \theta= F_{\theta}^{(\alpha)},}
\end{equation}
where
\begin{equation} \label{def of F_theta^alpha}
	F_{\theta}^{(\alpha)} =  \hat{Z}^{\alpha} F_\theta + \frac{\mu}{1+t} \partial_{t} Z^{\alpha} \theta-\hat Z^{\alpha}\left(\frac{\mu}{1+t} \partial_{t} \theta\right).
\end{equation}
Similarly, by applying $ \hat{Z}^{\alpha} $ to \eqref{wave eqn of u} we have 
\begin{equation} \label{eqn of Z^alpha u}
	\boxed{\Box Z^{\alpha} u + \frac{\mu}{1+t} \partial_{t} Z^{\alpha} u - \frac{\mu}{(1+t)^{2}} Z^{\alpha} u = F_u^{(\alpha)}},
\end{equation}
where
\begin{equation} \label{def of F_u^alpha}
	F_u^{(\alpha)} =  \hat{Z}^{\alpha} F_u + \partial_t\left(\frac{\mu}{1+t}Z^\alpha u\right)-\hat Z^\alpha\partial_t\left(\frac{\mu}{1+t}u\right).
\end{equation}

\subsection{Useful inequalities}
In this subsection we state with several inequalities, which will be useful in our proof. Throughout this paper, for any function $\Phi$ on $\R^2$ we use the notation 
\begin{equation}
 \left\| \Phi \right\| = \left\| \Phi \right\|_{L^{2} ( \R^{2} )}, \ \ |\Phi |_{\infty} = \left\| \Phi \right\|_{L^{\infty} ( \R^{2} )}.
\end{equation}

The following Klainerman-Sideris estimates were proved in \cite{KS} for the $3$-$D$ case. Nevertheless, the proof therein applies also to $2$-$D$ without modification. We refer to Lemma 3.3 in \cite{Lei2016}. 
\begin{lem}[\cite{KS}, Lemma 2.3] \label{lempointwise KS estimate}
	Denote $r=|x|$. For $\Phi\in C^\infty(\mathbb{R}_+\times\mathbb{R}^2)$, we have that
	\begin{equation}
		\left|(t-r)\partial_t^2\Phi\right|\lesssim\left|\partial Z^{\le1}\Phi\right|+r|\Box\Phi|,
		\label{pointwise KS estimate 1}
	\end{equation}
	\begin{equation}
		\left|(t-r)\Delta\Phi\right|+\left|(t-r)\partial_t\nabla\Phi\right|\lesssim\left|\partial Z^{\le1}\Phi\right|+t|\Box\Phi|.
		\label{pointwise KS estimate 2}
	\end{equation}
\end{lem}

Based on Lemma \ref{lempointwise KS estimate}, we may further obtain the next weighted $L^2-L^2$ inequality, which will contribute to the weighted $ L^{\infty}- L^{2} $ estimates below (Lemma \ref{lem1+t-x} and \ref{lem(1+t-x)3/2}).
\begin{lem}[\cite{KS}, Lemma 3.1]\label{Klainerman-Sideris estimate}
	Let $\sigma(t,x)=1+t-|x|$ and $\alpha\in [0,\infty)$. We have the following weighted $L^2-L^2$ estimate for smooth $\Phi$ with $\operatorname{supp}_x\Phi\subset B(0,1+t)$:
	\begin{equation}
		\left\|\sigma^{1+\alpha}\partial^2\Phi\right\|\lesssim_\alpha\left\|\sigma^\alpha\partial Z^{\le1}\Phi\right\|+(1+t)\|\sigma^\alpha\Box\Phi\|.
		\label{KS estimate}
	\end{equation}
\end{lem}
\begin{proof}
	From \eqref{pointwise KS estimate 1}\eqref{pointwise KS estimate 2}\eqref{equivalent form of DZ^alpha} and $\operatorname{supp}_x\Phi\subset B(0,1+t)$, we can deduce that
	\begin{equation}
		\begin{aligned}
			&\sigma^{1+\alpha}\left(\left|\Delta\Phi\right|+\left|\partial_t\partial\Phi\right|\right)
			\lesssim\sigma^\alpha\left|\partial Z^{\le1}\Phi\right|+(1+t)\sigma^\alpha\left|\Box\Phi\right|.
		\end{aligned}
	\end{equation}	
	This implies
	\begin{equation}
		\begin{aligned}
			&\left\|\sigma^{1+\alpha}\Delta\Phi\right\|+\left\|\sigma^{1+\alpha}\partial_t\partial\Phi\right\|
			\lesssim\left\|\sigma^\alpha\partial Z^{\le1}\Phi\right\|+(1+t)\left\|\sigma^\alpha\Box\Phi\right\|.
		\end{aligned}
		\label{improved KS estimate, incomplete}
	\end{equation}	
We still need to bound $\|\sigma^{1+\alpha}\nabla^2\Phi\|$. Integrating by parts and using the fact $|\nabla\sigma|_\infty\lesssim1$, we have that
	\begin{equation}
		\begin{aligned}
			\sum\limits_{ |\gamma| = 2} \left\|\sigma^{1+\alpha}  \nabla^{\gamma} \Phi \right\| ^{2} &\leq \left\|\sigma^{1+\alpha} \Delta\Phi \right\| ^{2} + \frac{1}{2} \sum\limits_{ |\gamma| = 2} \left\|\sigma^{1+\alpha}  \nabla^{\gamma} \Phi \right\| ^{2} + C\left\|\sigma ^\alpha\nabla \Phi \right\| ^{2} .
		\end{aligned}
	\end{equation}
	Hence we obtain 
	\begin{equation}
		\begin{aligned}
			\left\|\sigma^{1+\alpha}  \nabla^{2}\Phi \right\| & \lesssim  \left\|\sigma^{1+\alpha} \Delta\Phi \right\| + \left\|\sigma^\alpha \nabla \Phi \right\| \overset{\eqref{improved KS estimate, incomplete}}{\lesssim} \left\|\sigma^\alpha\partial Z^{\le1}\Phi\right\|+(1+t)\left\|\sigma^\alpha\Box\Phi\right\|.
		\end{aligned}
	\end{equation}
\end{proof}

In the rest of this section, we aim to derive three weighted $ L^{\infty}- L^{2} $ Sobolev imbedding inequalities, which will be useful in control several auxiliary energies defined in the next section, and accordingly make a valuable contribution to closing the bootstrap argument.

\begin{lem} \label{lem1+t-x}
	Suppose $ f\in C_c(\R^2)$. We have that
	\begin{equation}
		\begin{aligned}
			|  \sigma^{ \frac{1}{2} } f |_{\infty } \lesssim \left\| f \right\|^{ \frac{1}{2} } \left( \left\| \sigma \partial^{2} f  \right\|^{ \frac{1}{2} } + \left\| \partial f \right\|^{ \frac{1}{2} }   \right)+\|\sigma^{\frac{1}{2}}\partial f\|. 
		\end{aligned}
	\end{equation}
\end{lem}
\begin{proof}
	Note that $ | \partial \sigma | \lesssim 1 $ and $ | \partial^{2} \sigma | \lesssim \frac{1}{|x|} $. From the fundamental theorem of calculus, we have that
	\begin{equation}
		\begin{aligned}
			|  \sigma f^{2} |& \lesssim \left\|  \partial_{x_{1}} \partial_{x_{2}}  (\sigma f^{2}) \right\| _{L^{1}( \R^{2} )} \\
			& \lesssim \left\|  \frac{f^{2}}{|x|} \right\|_{L^{1}} + \left\|   \nabla (f^{2})  \right\| _{L^{1}} + \left\|  \sigma\nabla^2 (f^{2})  \right\|_{L^{1}}\\
			\overset{\eqref{Hardy inequality}}&{\lesssim}\|\nabla(f^2)\|_{L^1}+\|\sigma f\nabla^2f\|_{L^1}+\|\sigma\nabla f\cdot\nabla f\|_{L^1}\\
			\overset{\text{H\"older}}&{\lesssim}\left\| f \right\| \left( \left\| \sigma \partial^{2} f  \right\| + \left\| \partial f \right\|  \right)+\|\sigma^{\frac{1}{2}}\partial f\|^2.
		\end{aligned}
	\end{equation}
	This completes the proof. 
\end{proof}

\begin{lem} \label{lem(1+t-x)3/2}
	For $ f\in C_c^\infty(B_{1+t}(0)) $ , we have that
	\begin{equation}
		\begin{aligned}
			|  \sigma^{ \frac{3}{2} } f |_{\infty } & \lesssim \left\| \sigma f \right\|^{ \frac{1}{2} } \cdot ( \left\| \sigma^{2} \partial^{2} f  \right\|^{ \frac{1}{2} } + \left\| \sigma \partial f \right\|^{ \frac{1}{2} }   )  + \| \sigma^{ \frac{1}{2} } f \|. 
		\end{aligned}
	\end{equation}
	
\end{lem}

\begin{proof}
Note that $ | \partial \sigma^3 | \lesssim \sigma^2 $ and $ | \partial^{2} \sigma^3 | \lesssim \sigma + \frac{\sigma^2}{|x|} $. Again from the fundamental theorem of calculus, we have that
	\begin{equation}
		\begin{aligned}
			|  \sigma^{3} f^{2} | \lesssim &\left\|  \partial_{x_{1}} \partial_{x_{2}}  (\sigma^{3} f^{2}) \right\| _{L^{1}}  \lesssim \left\|  \partial^{2}  (\sigma^{3}) f^{2} \right\|_{L^{1}} + \left\|  \sigma^{2} \partial (f^{2})  \right\| _{L^{1}} + \left\|  \sigma^{3}\partial_{x_{1}} \partial_{x_{2}} (f^{2})  \right\|_{L^{1}}\\
			 \lesssim &\left\| (\sigma + \frac{\sigma^{2}}{|x|})  f^{2}  \right\|_{L^{1}} + \left\|  \sigma^{2} \partial f \cdot f \right\| _{L^{1}} + \left\|  \sigma^{3} ( \partial_{x_{1}} f \cdot \partial_{x_{2}} f + f \cdot	\partial_{x_{1}} \partial_{x_{2}} f )\right\|_{L^{1}}\\
			\overset{\eqref{Hardy inequality}}{\lesssim}&\left\| \sigma^{ \frac{1}{2} } f \right\|^{2} + \left\| \partial (\sigma^{2} f^{2})   \right\|_{L^{1}} + \left\| \sigma \partial f \right\| \cdot \left\| \sigma f \right\| + \left\|  \sigma^{3}  \partial_{x_{1}} f \cdot \partial_{x_{2}} f \right\|_{L^{1}} \\
&+ \left\| \sigma^{2} \partial^{2} f \right\| \cdot \left\| \sigma f \right\|\\
			 \lesssim &\left\| \sigma^{ \frac{1}{2} } f \right\|^{2} +\left\| \sigma \partial f \right\| \cdot \left\| \sigma f \right\| + \left\| \sigma^{2} \partial^{2} f \right\| \cdot \left\| \sigma f \right\| + \left\|  \sigma^{3}  \partial_{x_{1}} f \cdot \partial_{x_{2}} f \right\|_{L^{1}},
		\end{aligned}
	\end{equation}
the last term can be estimated by using integration by parts 
	\begin{equation}
		\begin{aligned}
\left\|  \sigma^{3}  \partial_{x_{1}} f \cdot \partial_{x_{2}} f \right\|_{L^{1}} 
& \leq \int_{ \mathbb{R}^{2} } \sigma^{3}[ (\partial_{x_{1}} f)^{2} + (\partial_{x_{2}} f)^{2}]  dx_{1} dx_{2} \\		
	& \lesssim 	\left\|  \sigma^{2}  \partial f \cdot f \right\|_{L^{1}} + 	\left\|  \sigma^{3}  \partial^{2} f \cdot f \right\|_{L^{1}}\\
	& \lesssim \left\| \sigma \partial f \right\| \cdot \left\| \sigma f \right\| +  \left\|\sigma^{2} \partial^{2} f \right\|  \cdot \left\| \sigma f \right\|.
		\end{aligned}
	\end{equation}			
Combining the above estimates gives
	\begin{equation}
		\begin{aligned}
			|  \sigma^{3} f^{2} | \lesssim \left\|\sigma f \right\| \cdot  ( \left\| \sigma^{2} \partial^{2} f \right\| + \left\| \sigma \partial f \right\| ) + \| \sigma^{\frac{1}{2}} f \|^{2}.
		\end{aligned}
	\end{equation}
This completes the proof.

\end{proof}

\begin{lem} \label{KSlow}
	For any sufficiently regular function $f$ on $\R^{2}$, we have 
	\begin{equation}
		\begin{aligned}
			|f(x)| \lesssim |x|^{ -\frac{1}{2} } \left\|\partial^{\le1}\Omega^{\le1}f \right\|. 
		\end{aligned}
	\end{equation}
\end{lem}
\begin{proof}
	Using the Sobolev embedding $H^1(\mathbb{S}^1)\hookrightarrow L^\infty(\mathbb{S}^1)$, and the fundamental theorem of calculus, it is east to get
	\begin{equation}
		\begin{aligned}
			| f(r,\theta) |^{2} & \lesssim \int_{0}^{2\pi} ( |\Omega f |^{2} + | f |^{2} ) d \theta \\
			& \lesssim r^{-1} \int_{r}^{+\infty}  \int_{0}^{2\pi} (  | \partial_{r} \Omega f \cdot \Omega f  | +  | \partial_{r} f \cdot  f  | ) \rho d \theta d\rho\\
			&\lesssim r^{ -1 } \left\| \partial_{r} \Omega^{ \leq 1 } f \right\| \cdot \left\|  \Omega^{ \leq 1 } f \right\|. 
		\end{aligned}
	\end{equation}	
	This yields the desired estimate.
	
\end{proof}

\section{Reduction to a bootstrap argument}
\subsection{Main result in the reformulated system}
We state the main result to the reformulated system of $(u,\theta)$. 
\begin{thm}
\label{main thm in u,theta system}
 Suppose $0 < \mu \leq 1 $, and the initial data $(u_0,\theta_0)\in C_c^\infty(\R^2)$ satisfies:
 \begin{equation}
 	\operatorname{curl} u_{0} = 0,
 	\label{initial condition, vorticiti free}
 \end{equation}
 \begin{equation}
 	\operatorname{supp}(u_0,\theta_0)\subset \overline{B_{1/2}(0)}.
 	\label{initial condition, support}
 \end{equation}
 Then there exist constants $0<\epsilon_0\ll1$ and $C>0$, such that for all $ \epsilon\in(0,\epsilon_0] $, $ ( \ref{2} ) $ admits a $ C_c^{\infty} $ solution $ (\theta,u) $ in $ [ 0,T_{\epsilon} ] $, where 
 \begin{equation} 
 	\begin{aligned}
 		T_{\epsilon} =
 		\begin{split}
 			\left\lbrace
 			\begin{array}{lr}
 				 C \epsilon^{ - \frac{2}{1-\mu} }, \ \ 0 < \mu < 1 \\
 				C \exp ( \epsilon^{ - 1 } ) , \ \ \mu = 1.
 			\end{array}	
 			\right. 
 		\end{split}	
 	\end{aligned}
 \end{equation}
The quantity $\epsilon_0,C$ only depend on $u_0,\theta_0$ and the constants $\mu,\gamma$. Moreover, we have that
\begin{equation}
	|\theta|_{\infty}+|u|_{\infty}\lesssim \frac{1}{M},
	\label{Hk norm of the solution}
\end{equation}
where $M$ is a fixed constant large enough satisfying \eqref{M}.
\end{thm}

Assuming Theorem \ref{main thm in u,theta system} holds true, we prove Theorem \ref{thm1} next. 
\begin{proof}[Proof of the main theorem]We have smooth $(\theta,u)$ on $[0,T_\epsilon]$ which solves the system \eqref{2}. To obtain a smooth solution to \eqref{Euler eqns} via
\begin{equation}
	\rho=\left[(\gamma-1)\theta+1\right]^{\frac{1}{\gamma-1}},
\end{equation}
we need to show that $(\gamma-1)\theta+1$ is bounded away from $0$. This is ensured by
\begin{equation}
	|(\gamma-1)\theta|_\infty\overset{\eqref{Hk norm of the solution}}{\lesssim}\frac{1}{M}<\frac{1}{2}.
\end{equation}
Therefore, we have a smooth solution to the original system over $[0,T_\epsilon]$. 
    
\end{proof}
\subsection{Bootstrap argument}
To prove Theorem \ref{main thm in u,theta system}, we run a bootstrap argument. 

\subsubsection{Time-weighted energies}
Recall that $Z$ is defined in \eqref{def of Z}, and $Z^\alpha=Z_0^{\alpha_0}Z_1^{\alpha_1}Z_2^{\alpha_2}Z_3^{\alpha_3}Z_4^{\alpha_4}$. For any function $\Phi(t,x)$, we define a time-weighted energy associated with the vector fields $Z$ by
\begin{equation}
\begin{aligned}
E_{k} [ \Phi ] ( t) = (1+t)^{\frac{\mu}{2}} \sum\limits_{ 0 \leq |\alpha| \leq k} \left\| \partial Z^{\alpha} \Phi ( t) \right\| + (1+t)^{\frac{\mu}{2}-1} \sum\limits_{ 0 \leq |\alpha| \leq k} \left\| Z^{\alpha} \Phi ( t) \right\| ,
\label{def of Ek}
\end{aligned}
\end{equation}
where $k$ is a fixed positive number, and $ \left\| \cdot \right\| $ stands for the $ L^{2}_{x} $ norm on $\R^{2}$. If $\Phi$ is smooth and supported in $B_{1+t}(0)$, it is easy to get by \eqref{Poincare ineq.}
\begin{equation}
	E_k[\Phi](t)\le 2(1+t)^{\frac{\mu}{2}} \sum\limits_{ 0 \leq |\alpha| \leq k} \left\| \partial Z^{\alpha} \Phi ( t) \right\|\le 2E_k[\Phi](t). 
	\label{equivalent norm of Ek}
\end{equation}

In order to close the bootstrap, we introduce another three auxiliary weighted energies
\begin{equation}
\eta [\Phi] (t) \triangleq  (1+t)^{ \frac{1+\mu}{2} } | Z^{\leq 2}  \partial \Phi (t) |_{\infty },
\label{def of eta}
\end{equation}
\begin{equation}
\chi_{k} [\Phi] (t) \triangleq (1+t)^{ \frac{\mu}{2} } \sum\limits_{ 0 \leq |\alpha| \leq k-1} \left\| \sigma \partial^{2} Z^{\alpha} \Phi \right\|,
\end{equation}

\begin{equation}
	\tilde{\chi}_{k} [\Phi] (t) \triangleq (1+t)^{ \frac{\mu}{2} } \sum\limits_{ 0 \leq |\alpha| \leq k-3} \left\| \sigma^{2} \partial^{4} Z^{\alpha} \Phi \right\|,
\end{equation}
where $\sigma(t,x)=1+t-|x|$. Obviously we have 
\begin{equation}
	\chi_k[\Phi](t)\le (1+t)E_k[\Phi](t). 
\end{equation}

Throughout the paper, we use
\begin{equation}
E_{k} [ \Phi_{1},\Phi_{2} ] \triangleq E_{k} [ \Phi_{1} ]+ E_{k} [\Phi_{2} ]
\end{equation}
to denote the sum of energies. Similar notations apply to $\eta$, $\chi_k$, and $\tilde\chi_k$.

\subsubsection{Bootstrap procedure}

From the fact that $u_0,\theta_0\in C_c^\infty(\R^2)$ and \eqref{properties of g}, we know that there exists a constant $C_0$ which relies on the initial data $u_0,\rho_0$ , such that
\begin{equation}
	E_5[\theta,u](0)+\eta[\theta,u](0)\le C_0\epsilon.
\end{equation}

At this stage, we introduce a large but fixed constant $M$, which will be used to absorb $C_0$ and the universal constants appearing in the proof. Based on $M$, we will choose $\epsilon_0$ small enough. In fact, we have the following hierarchy:
\begin{equation}\label{M}
    1\ll M\ll\exp{M}\ll\epsilon_0^{-1}. 
\end{equation}
The main theorem holds true for any $\epsilon<\epsilon_0$. 

For the case $\mu\in(0,1)$, we define a large time $T_\epsilon$ implicitly via the following identity:
\begin{equation}
	(1+T_\epsilon)^{\frac{1-\mu}{2}}\epsilon=\frac{1}{M^2}.
\end{equation}
If $\mu=1$, we define $T_\epsilon$ by
\begin{equation}
	\ln(1+T_\epsilon)\epsilon=\frac{1}{M}. 
\end{equation}

\begin{rem}
	In each case, the following inequality is always true:
	\begin{equation}
		\label{usage of T_epsilon}
		(1+T_\epsilon)^{\frac{1-\mu}{2}}\epsilon\le\frac{1}{M^2}. 
	\end{equation}
\end{rem}

\paragraph{Bootstrap assumptions}We propose the bootstrap assumption for $\theta,u$:
\begin{equation}
    \boxed{\label{bootstrap assumption}
    E_5[\theta,u](t)+\eta[\theta,u](t)\le M \epsilon.  }
\end{equation}

Now we can state the key bootstrap lemma. 
\begin{lem}\textbf{[Bootstrap lemma]}
    \label{bootstrap lemma}
    There exist positive quantities $\epsilon_0,M$ depending only on $u_0,\theta_0,\mu,\gamma$, such that for any $[0,T]\subset[0,T_\epsilon]$, if the solution $(\theta,u)$ exists on $[0,T]$ and satisfies the bounds \eqref{bootstrap assumption} for $t\in[0,T]$, then the solution actually satisfies improved bounds:
    \begin{equation}
        \label{improved bound}
    E_5[\theta,u](t)+\eta[\theta,u](t)\le \frac{1}{2}M\epsilon. 
    \end{equation}
\end{lem}

The proof of this lemma will be given at the end of section \ref{section estimate of E5theta}. Assuming the bootstrap lemma \ref{bootstrap lemma} is true, we prove Theorem \ref{main thm in u,theta system}, which further implies our main result. 

\begin{proof}[Proof of Theorem \ref{main thm in u,theta system}]Define
\begin{equation}
	\begin{aligned}
	T_{*} = \sup \left\{ T \leq \min( T_{\epsilon}, T_{max}):\eqref{bootstrap assumption}\text{ holds on }[0,T]\right\},
	\end{aligned}
\end{equation}		
where $T_{max}$ is the lifespan of the solution. We prove that $T_*=T_\epsilon$. If this is not true, we have $T=T_{max}<T_\epsilon<\infty$. From the bootstrap lemma \ref{bootstrap lemma} and the standard bootstrap argument, we know that \eqref{bootstrap assumption} holds on $[0,T_{max})$. However, $\eqref{bootstrap assumption}$ precludes all the scenarios described in Theorem \ref{thmmajda2}. This is a contradiction. Consequently, it must hold that $T_*=T_\epsilon\le T_{max}$. In other words, smooth solution $(\theta,u)$ exists on $[0,T_\epsilon]$.  
	
The inequality \eqref{Hk norm of the solution} will be proved in Lemma \ref{L^infty for Z4Phi, lem}. 
\end{proof}

The rest of this paper is devoted to proving Lemma \ref{bootstrap lemma}.

\section{Forcing estimate}
In this section, we assume that $t\le T_\epsilon$, and the bootstrap assumption \eqref{bootstrap assumption} remains true on $[0,t]$. Firstly we give a basic $L^\infty$ estimate, which is an immediate consequence of bootstrap assumption. 
\begin{lem}
	\label{L^infty for Z4Phi, lem}
	For $\Phi\in\{\theta,u\}$ and $|\alpha|\le4$, we have that
	\begin{equation}
		\label{L^infty for Z4Phi}
		|Z^\alpha\Phi|_\infty\lesssim M^{-1}. 
	\end{equation}
\end{lem}
\begin{proof}
	By the Gagliardo-Nirenberg interpolation inequality \eqref{GN inequality}, we have
	\begin{equation}
		\begin{aligned}
		|Z^\alpha\Phi|_\infty & \lesssim\|Z^\alpha\Phi\|^{\frac{1}{2}}\|\partial^2Z^\alpha\Phi\|^{\frac{1}{2}}\lesssim(1+t)^{\frac{1-\mu}{2}}E_5[\theta,u] \\
		& \lesssim(1+t)^{\frac{1-\mu}{2}}M\epsilon\overset{\eqref{usage of T_epsilon}}{\lesssim} M^{-1}. 
	\end{aligned}
	\end{equation}
	This completes the proof.
\end{proof}

The next lemma deals with the estimate of error terms generated from the linear damping in the righthand of \eqref{eqn of Z^alpha theta}\eqref{eqn of Z^alpha u}. This key ingredient is that these terms admits good enough decay by applying vector fields $\widehat{Z}^\alpha$.
\begin{lem}
	\label{estimate of F^alpha-Z^alpha F}
	For $\Phi\in\{\theta,u\}$ and $|\alpha|\le 5$, we have that
	\begin{equation}
		\left|F_\Phi^{(\alpha)}-\hat Z^\alpha F_\Phi\right|\lesssim\mathbbm{1}_{\alpha_0+\alpha_3>0}\frac{\mu}{(1+t)^2}\left|\partial Z^{\le|\alpha|-1}\Phi\right|.
		\label{estimate of F^alpha-Z^alpha F, ineq}
	\end{equation}
\end{lem}
\begin{proof}
	(1) $\Phi = \theta$: From \eqref{commute Z and damping, for theta}, and the fact that
	\begin{equation}
		\left|\hat Z^\gamma\left(\frac{1}{(1+t)^k}\right)\right|\lesssim\frac{1}{(1+t)^k},
	\end{equation}
	we have
	\begin{equation}
		\begin{aligned}
			\left|F_\theta^{(\alpha)}-\hat Z^\alpha F_\theta\right|\lesssim\mathbbm{1}_{\alpha_0>0}\frac{\mu}{(1+t)^2}\left|\partial Z^{<\alpha}\theta\right|+\mathbbm{1}_{\alpha_3>0}\frac{\mu}{(1+t)^2}\sum_{j=0}^{\alpha_3-1}\sum_{\beta\le\alpha-(j+1)e_3}\left|\hat Z^\beta\partial_t Z^{je_3}\theta\right|.
		\end{aligned}
	\end{equation}
	Since $\hat Z_3=Z_3+2$, we have
	\begin{equation}
		\hat Z^\beta=\sum_{l=0}^{\beta_3}\binom{\beta_3}{l}2^lZ^{\beta-le_3}=Z^{\le\beta}. 
	\end{equation}
	Thus, the mixed derivatives of $\theta$ can be bounded by
	\begin{equation}
		\begin{aligned}
			\left|\hat Z^\beta\partial_t Z^{je_3}\theta\right|&\lesssim\left|Z^{\le\beta}\partial_t Z^{je_3}\theta\right|\overset{\eqref{commutator}}{\lesssim}\left|\left(\partial t Z^{\le\beta}Z^{je_3}+\partial Z^{<|\beta|}\right)Z^{je_3}\theta\right|\\
			\overset{\eqref{cononical filtration}}&{\lesssim}\left|\partial Z^{\le|\beta|+j}\theta\right|
			\overset{|\beta|+j\le|\alpha|-1}{\lesssim}\left|\partial Z^{\le|\alpha|-1}\theta\right|.
		\end{aligned}
	\end{equation}
	Finally we have that
	\begin{equation}
		\begin{aligned}
			\left|F_\theta^{(\alpha)}-\hat Z^\alpha F_\theta\right|&\lesssim\mathbbm{1}_{\alpha_0>0}\frac{\mu}{(1+t)^2}\left|\partial Z^{<\alpha}\theta\right|+\mathbbm{1}_{\alpha_3>0}\frac{\mu}{(1+t)^2}\left|\partial Z^{\le|\alpha|-1}\theta\right|\\
			&\lesssim\mathbbm{1}_{\alpha_0+\alpha_3>0}\frac{\mu}{(1+t)^2}\left|\partial Z^{\le|\alpha|-1}\theta\right|.
		\end{aligned}
	\end{equation}
	
	(2) $\Phi = u$: Noting that
	\begin{equation}
		\begin{aligned}
			\left|F_u^{(\alpha)}-\hat Z^\alpha F_u\right|&\le\left|\frac{\mu}{(1+t)^2}Z^\alpha u-\hat Z^\alpha\left(\frac{\mu}{(1+t)^2}u\right)\right|+\left|\frac{\mu}{1+t}\partial_tZ^\alpha u-\hat Z^\alpha\left(\frac{\mu}{1+t}\partial_t u\right)\right|.
		\end{aligned}
	\end{equation}
	Making computation as we did in deriving the bound for $\left|F_\theta^{(\alpha)}-\hat Z^\alpha F_\theta\right|$, one can see that the second term is bounded by $\mathbbm{1}_{\alpha_0+\alpha_3>0}\frac{\mu}{(1+t)^2}\left|\partial Z^{\le|\alpha|-1}u\right|$. Using \eqref{commute Z and damping, for u}, one can bound the first term as
	\begin{equation}
		\left|\frac{\mu}{(1+t)^2}Z^\alpha u-\hat Z^\alpha\left(\frac{\mu}{(1+t)^2}u\right)\right|\lesssim\mathbbm{1}_{\alpha_0+\alpha_3>0}\frac{\mu}{(1+t)^3}\left|Z^{\le|\alpha|-1}u\right|. 
	\end{equation}
	Finally, from \eqref{Poincare ineq.}, we have that
	\begin{equation}
		\left|F_u^{(\alpha)}-\hat Z^\alpha F_u\right|\lesssim\mathbbm{1}_{\alpha_0+\alpha_3>0}\frac{\mu}{(1+t)^2}\left|\partial Z^{\le|\alpha|-1}u\right|. 
	\end{equation}
	This completes the proof.
\end{proof}

\begin{lem} \label{lemF^alpha-Z^alphaF}
	For $\Phi\in\{\theta,u\}$ and $|\alpha|\le 5$, we have that
	\begin{equation}
		\|F_{\Phi}^{(\alpha)}-\hat Z^\alpha F_\Phi\|\lesssim\mathbbm{1}_{\alpha_0+\alpha_3>0}\mu(1+t)^{-2-\frac{\mu}{2}}E_{|\alpha|-1}[\Phi].
	\end{equation}

\end{lem}
\begin{proof}
	It is a direct result of Lemma \ref{estimate of F^alpha-Z^alpha F} and the definition \eqref{def of Ek} of $E_k$. 
\end{proof}

To control the $L^2$-norms of $F_{\theta,1}$ and $F_{u,1}$, we need the following two auxiliary lemmas. 

\begin{lem}
	\label{auxiliary asymmetric estimate}
	Suppose $\Phi_i\in\{\theta,u\}$, $i=1,2,3$, then we have
	\begin{equation}
		\|Z^{\le5}\partial\Phi_1 Z^{\le2}\partial\Phi_2\|\lesssim(1+t)^{-\frac{1}{2}-\mu}\eta[\theta,u]E_5[\theta,u],
		\label{estimate of Z5D*Z2D}
	\end{equation}
	\begin{equation}
		\|Z^{\le2}\Phi_1 Z^{\le5}\partial\Phi_2\|+\|Z^{\le5}\Phi_1 Z^{\le2}\partial\Phi_2\|\lesssim(1+t)^{\frac{1}{2}-\mu}\eta[\theta,u]E_5[\theta,u],
		\label{estimate of Z2*Z5D+Z5*Z2D}
	\end{equation}
	\begin{equation}
		\|Z^{\le2}\Phi_1 Z^{\le2}\partial\Phi_2Z^{\le5}\partial\Phi_3\|+\|Z^{\le5}\Phi_1 Z^{\le2}\partial\Phi_2Z^{\le2}\partial\Phi_3\|\lesssim(1+t)^{-\frac{3}{2}\mu}\eta[\theta,u]^2E_5[\theta,u].
		\label{estimate of Z2*Z2D*Z5D+Z5*Z2D*Z2D}
	\end{equation}
\end{lem}
\begin{proof}
	These inequalities can be proved by using H\"older's inequality, Lemma \ref{Poincare ineq.}, and the definitions \eqref{def of Ek}\eqref{def of eta} of $E_k$ and $\eta$. 
\end{proof}

\begin{lem}
	\label{L2 norm of multiplications}
	Suppose $\Phi_i\in\{\theta,u\}$, $i=1,2,3$. For any multi-index $\alpha\in\mathbb{N}^5$ with $|\alpha|\le5$, we have that
	\begin{equation}
		\label{L2 of Z5(Phi1 DPhi2)}
		\|\hat Z^\alpha(\Phi_1\partial\Phi_2)\|\lesssim(1+t)^{\frac{1}{2}-\mu}\eta[\theta,u]E_5[\theta,u],
	\end{equation}
	\begin{equation}
		\label{L2 of Z5(DPhi1 DPhi2)}
		\|\hat Z^\alpha(\partial\Phi_1\partial\Phi_2)\|\lesssim(1+t)^{-\frac{1}{2}-\mu}\eta[\theta,u]E_5[\theta,u],
	\end{equation}
	\begin{equation}
		\label{L2 of Z5(Phi1 DPhi2 DPhi3)}
		\|\hat Z^\alpha(\Phi_1\partial\Phi_2\partial\Phi_3)\|\lesssim(1+t)^{-\frac{3}{2}\mu}\eta[\theta,u]^2E_5[\theta,u].
	\end{equation}
\end{lem}
\begin{proof}
	By \eqref{relation between Z and Zhat} and using the pigeonhole principle, we have
	\begin{equation}
		\sum_{|\beta|+|\gamma|\le 5}=\sum_{\substack{|\beta|+|\gamma|\le 5 \\ |\beta|\le2}}+\sum_{\substack{|\beta|+|\gamma|\le 5 \\ |\gamma|\le2}},
	\end{equation}
	\begin{equation}
		\sum_{|\beta|+|\gamma|+|\delta|\le 5}=\sum_{\substack{|\beta|+|\gamma|+|\delta|\le 5 \\ |\beta|+|\gamma|\le2}}+\sum_{\substack{|\beta|+|\gamma|+|\delta|\le 5 \\ |\gamma|+|\delta|\le2}}+\sum_{\substack{|\beta|+|\gamma|+|\delta|\le 5 \\ |\alpha|+|\delta|\le2}}.
	\end{equation}
	Then the conclusions follow immediately from the Leibniz rule and Lemma \ref{auxiliary asymmetric estimate}. 
\end{proof}

Now we are ready to estimate the $L^2$-norms of $\hat{Z}^{\alpha}F_{\theta,1}$ and $\hat{Z}^{\alpha}F_{u,1}$. 

\begin{cor} \label{ZFtheta1}
	For $ \Phi \in \{\theta,u\} $ and $|\alpha|\le5$, we have that
	\begin{equation}
		\left\|\hat{Z}^{\alpha} F_{\Phi,1} \right\|\lesssim(1+t)^{-\frac{1}{2}-\mu}\eta[\theta,u]E_5[\theta,u].
		\label{estimate of Z^alpha F1}
	\end{equation}
\end{cor}
\begin{proof}
	$F_{\Phi,1}$ consists of terms in the form of $\Phi_1\partial\Phi_2$, $\partial\Phi_1\partial\Phi_2$, and $\Phi_1\partial\Phi_2\partial\Phi_3$.
	Note that $|\hat Z^\alpha[(1+t)^{-1}]|\lesssim(1+t)^{-1}$ and $\hat Z^\alpha=Z^{\le|\alpha|}$. Using the Leibniz rule and Lemma \ref{L2 norm of multiplications} we can obtain that
		\begin{equation}
		\left\|\hat{Z}^{\alpha} F_{\theta,1} \right\|\lesssim(1+t)^{-\frac{1}{2}-\mu}\eta[\theta,u]E_5[\theta,u]+(1+t)^{-\frac{3}{2}\mu}\eta[\theta,u]^2E_5[\theta,u],
	\end{equation}
	\begin{equation}
		\left\|\hat{Z}^{\alpha} F_{u,1} \right\|\lesssim(1+t)^{-\frac{1}{2}-\mu}\eta[\theta,u]E_5[\theta,u].
	\end{equation}
Furthermore, using \eqref{usage of T_epsilon} and the bootstrap assumption \eqref{bootstrap assumption}, we have
\begin{equation}
(1+t)^{\frac{1-\mu}{2}}\eta\le(1+T_\epsilon)^{\frac{1-\mu}{2}}M\epsilon\le M^{-1}\ll1,
\end{equation}
and this proves the corollary. 	
\end{proof}

To control the $L^2$-norms of $F_{\theta,2}$ and $F_{u,2}$, we need the following lemma. To simplify the notation a bit, hereinafter we shall write the energies by $E_k,\chi_k,\tilde\chi_k,\eta$ instead of $E_k[\theta,u]$, $\chi_k[\theta,u]$, $\tilde\chi_k[\theta,u]$, $\eta[\theta,u]$. 

\begin{lem}
Suppose $\Phi_i\in\{\theta,u\}$, $i=1,2,3$. We have that
\begin{equation} \label{Z^leq2 Z^leq4 partial^2}
	\|Z^{\le2}\Phi_1Z^{\le4}\partial^2\Phi_2\|\lesssim(1+t)^{-\frac{1}{2}-\mu}\eta[\theta,u]\chi_5[\theta,u],
\end{equation}
\begin{equation} \label{Z^leq5 Z^leq2 partial^2}
	\left\| Z^{\leq 5} \Phi_{1} Z^{\leq 2} \partial^{2}  \Phi_{2}  \right\|\lesssim(1+t)^{-\frac{1}{2}-\mu}E_5[\theta,u]\left(\chi_5[\theta,u]+\tilde\chi_5[\theta,u]+E_5[\theta,u]\right).
\end{equation}
\end{lem}
\begin{proof}
The first inequality is proved by a direct calculation: 
\begin{equation}
	\begin{aligned}
		\|Z^{\le2}\Phi_1Z^{\le4}\partial^2\Phi_2\|&\le\left|\sigma^{-1} Z^{\le2}\Phi_1 \right|_\infty\|\sigma Z^{\le4}\partial^2\Phi_2\|\\
		\overset{\eqref{equivalent form of DZ^alpha}\eqref{Hardy inequality near boundary}}&{\lesssim}\left|\partial Z^{\le2}\Phi_1\right|_\infty\|\sigma\partial^2Z^{\le4}\Phi_2\|\\
		&\lesssim(1+t)^{-\frac{1}{2}-\mu}\eta[\theta,u]\chi_5[\theta,u]. 
	\end{aligned}
\end{equation}
In order to prove the second inequality, we shall divide into two parts:
\begin{equation}
	\left\| Z^{\leq5} \Phi_{1} Z^{\leq 2} \partial^{2}  \Phi_{2}  \right\|\le\left\| Z^{\leq 5} \Phi_{1} Z^{\leq 2} \partial^{2}  \Phi_{2}  \right\|_{ L^{2}( |x| \leq \frac{1+t}{2} ) } + \left\| Z^{\leq 5} \Phi_{1} Z^{\leq 2} \partial^{2}  \Phi_{2}  \right\|_{ L^{2}( |x| \geq \frac{1+t}{2} ) }.
\end{equation}
\textbf{Estimates of nonlinearities for $  |x| \leq \frac{1+t}{2} $.} Noting that $\operatorname{supp}(\theta,u)\subset B_{1/2+t}(0)$ and $\sigma^{\frac{1}{2}}\lesssim \sigma$ holds on the support of $(\theta,u)$. To control the first term, we make use of Lemma \ref{lem(1+t-x)3/2} to deduce that
\begin{equation}
	\begin{aligned}
		&| \sigma^{ \frac{3}{2} } Z^{\leq 2} \partial^{2}  \Phi_{2}  |_{\infty}\\
		\overset{\eqref{equivalent form of DZ^alpha}}{\lesssim}&\left\| \sigma \partial^{2}Z^{\leq 2}   \Phi_{2} \right\|^{ \frac{1}{2} } \cdot \left( \left\| \sigma^{2} \partial^{4}  Z^{\leq 2}  \Phi_{2}  \right\|^{ \frac{1}{2} } + \left\| \sigma \partial^3  Z^{\leq 2}\Phi_{2} \right\|^{ \frac{1}{2} } \right) + \left\| \sigma^{ \frac{1}{2} }\partial^{2}  Z^{\leq 2}   \Phi_{2} \right\|\\
		\lesssim& (1+t)^{-\frac{\mu}{2}}\left[\chi_{3}^{\frac{1}{2}}(\tilde{\chi}_5^{\frac{1}{2}}+\chi_{4}^{\frac{1}{2}})+\chi_{3}\right].
	\end{aligned}
\end{equation} 
Hence, The first term is estimated by
\begin{equation}
	\begin{aligned}
		\left\| Z^{\leq 5} \Phi_{1} Z^{\leq 2} \partial^{2}  \Phi_{2}  \right\|_{ L^{2}( |x| \leq \frac{1+t}{2} ) }&\le\left\|\sigma^{ -\frac{3}{2} } Z^{\leq 5} \Phi_{1} \right\|_{ L^{2}( |x| \leq \frac{1+t}{2} ) } \cdot | \sigma^{ \frac{3}{2} } Z^{\leq 2} \partial^{2}  \Phi_{2}  |_{\infty} \\
		&\lesssim(1+t)^{ - \frac{3}{2} } \left\|Z^{\leq 5} \Phi_{1} \right\|\cdot | \sigma^{ \frac{3}{2} } Z^{\leq 2} \partial^{2}  \Phi_{2}  |_{\infty}\\
		&\lesssim(1+t)^{-\frac{1}{2}-\mu}E_5\cdot\left[\chi_{3}^{\frac{1}{2}}(\tilde{\chi}_5^{\frac{1}{2}}+\chi_{4}^{\frac{1}{2}})+\chi_{3}\right].
	\end{aligned}
\end{equation}
\textbf{Estimates of nonlinearities for $  |x| \geq \frac{1+t}{2} $.} To bound the second term, we use Lemma \ref{KSlow} to deduce that
\begin{equation}
	\begin{aligned}
	|\sigma Z^{\le2}\partial^2\Phi_2\|_{L^\infty(|x|\ge\frac{1+t}{2})}&\lesssim(1+t)^{-\frac{1}{2}}\left\|\partial_r\left[\sigma Z^{\le3}\partial^2\Phi_2\right]\right\|^{\frac{1}{2}}\left\|\sigma Z^{\le3}\partial^2\Phi_2\right\|^{\frac{1}{2}}\\
		&\lesssim(1+t)^{-\frac{1}{2}}\left(\|Z^{\le3}\partial^2\Phi_2\|+\left\|\sigma Z^{\le3}\partial^3\Phi_2\right\|\right)^{\frac{1}{2}}\left\|\sigma Z^{\le3}\partial^2\Phi_2\right\|^{\frac{1}{2}}\\
		&\lesssim(1+t)^{-\frac{1+\mu}{2}}\chi_4^{\frac{1}{2}}(\chi_5+E_4)^{\frac{1}{2}}. 
	\end{aligned}
\end{equation}
Subsequently, the second term is bounded by
\begin{equation}
	\begin{aligned}
		\left\| Z^{\leq 5} \Phi_{1} Z^{\leq 2} \partial^{2}  \Phi_{2}  \right\|_{ L^{2}( |x| \geq \frac{1+t}{2} ) }&\le\left\|\sigma^{-1}Z^{\leq 5} \Phi_{1} \right\|\cdot \| \sigma Z^{\leq 2} \partial^{2}  \Phi_{2}  \|_{L^\infty(|x|\ge\frac{1+t}{2})}\\
		\overset{\eqref{Hardy inequality near boundary}}&{\lesssim}\|\partial Z^{\leq 5} \Phi_{1}\|\cdot (1+t)^{-\frac{1+\mu}{2}}\chi_4^{\frac{1}{2}}(\chi_5+E_4)^{\frac{1}{2}}\\
		&\lesssim(1+t)^{-\frac{1}{2}-\mu}\chi_4^{\frac{1}{2}}(\chi_5+E_4)^{\frac{1}{2}}E_5.
	\end{aligned}
\end{equation}
Combining the estimates of two parts together, we obtain
\begin{equation}
	\left\| Z^{\leq5} \Phi_{1} Z^{\leq 2} \partial^{2}  \Phi_{2}  \right\|\lesssim(1+t)^{-\frac{1}{2}-\mu}E_5\left(\chi_5+\tilde\chi_5+E_5\right). 
\end{equation}
\end{proof}

With the results of the previous lemmas in hand, we can now estimate $L^2$-norms of $\hat{Z}^{\alpha}F_{\theta,2}$ and $\hat{Z}^{\alpha}F_{u,2}$. 
\begin{cor}\label{ZFtheta2}
	For $ \Phi \in \{\theta,u\} $ and $|\alpha|\le4$, we have that	
	\begin{equation} 
		\label{estimate of Z^alpha F2}
		\begin{aligned}
			 \left\|  \hat{Z}^{\alpha} F_{\Phi,2} \right\|  
			& \lesssim  (1+t)^{ - \frac{1}{2} - \mu } \left[E_{5}\left(E_{5}+\chi_5+\tilde\chi_5\right) + \chi_{5} \eta\right] .
		\end{aligned}
	\end{equation}
\end{cor}
\begin{proof}
$F_{\theta,2}$ and $F_{u,2}$ consist of terms in the form of $\Phi_1\partial^2\Phi_2$ and $\Phi_1\Phi_2\partial^2\Phi_3$. Using the pigeonhole principle and the above lemmas, we have that
\begin{equation}
	\label{bilinear term in F_Phi,2}
	\begin{aligned}
		\|\hat Z^\alpha(\Phi_1\partial^2\Phi_2)\|&\lesssim\left\|Z^{\le2}\Phi_1Z^{\le4}\partial^2\Phi_2\right\|+\left\| Z^{\leq4} \Phi_{1} Z^{\leq 1} \partial^{2}  \Phi_{2}  \right\|\\
		&\lesssim  (1+t)^{ - \frac{1}{2} - \mu } \left[E_{5}\left(E_{5}+\chi_5+\tilde\chi_5\right) + \chi_{5} \eta\right]. 
	\end{aligned}
\end{equation}
Moreover, the cubic nonlinear term is bounded by
\begin{equation}
	\begin{aligned}
		\|\hat Z^\alpha(\Phi_1\Phi_2\partial^2\Phi_3)\|&\lesssim|Z^{\le4}\Phi_1|_\infty\|Z^{\le4}(\Phi_2\partial^2\Phi_3)\|\\
		\overset{\eqref{bilinear term in F_Phi,2}}&{\lesssim}M^{-1}(1+t)^{ - \frac{1}{2} - \mu } \left[E_{5}\left(E_{5}+\chi_5+\tilde\chi_5\right) + \chi_{5} \eta\right]. 
	\end{aligned}
\end{equation}
\end{proof}

The following lemma is helpful to the estimate of $ \tilde\chi_5 $.
\begin{lem} \label{lem(1+t-x)partial2FPhi}
	Recall that $\sigma(t,x)=1+t-|x|$. For $\Phi\in\{\theta,u\}$, we have the following weighted estimate for forcing terms:	
	\begin{equation} 
		\begin{aligned}
			\left\| \sigma \partial^{2} Z^{\leq 2} F_{\Phi} \right\|  \lesssim (1+t)^{-\mu -\frac{1}{2} } \eta[\theta,u]\cdot \left( \tilde{\chi}_{5} [\theta,u] + \chi_{4} [\theta,u ] \right).
		\end{aligned}
	\end{equation}
\end{lem}
\begin{proof}
	$F_{\Phi}$ consists of terms like $(1+t)^{-1}\Phi_{1} \partial \Phi_{2}$, $\partial\Phi_{1} \partial \Phi_{2}$, $\Phi_{1} \partial \Phi_{2} \partial \Phi_{3}$, $\Phi_{1} \partial^{2} \Phi_{2}$, and $\Phi_{1} \Phi_{2} \partial^{2} \Phi_{3}$. We estimate each kind of terms one by one. We first bound the $(1+t)^{-1}\Phi_{1} \partial \Phi_{2}$ terms by
	\begin{equation}
		\begin{aligned}
		&\ \ \ \left\| \sigma \partial^{2}Z^{\leq 2} ((1+t)^{-1} \Phi_{1} \partial \Phi_{2}) \right\|\\
		&\lesssim (1+t)^{-1} \left\| \sigma \partial^{2}Z^{\leq 2} ( \Phi_{1} \partial \Phi_{2}) \right\| \\
		&\lesssim (1+t)^{-1} \sum_{\tau\in S_2} \left\| \sigma Z^{\leq 2} (\partial^{2} \Phi_{\tau(1)} \partial \Phi_{\tau(2)}) \right\| + (1+t)^{-1} \left\|\sigma Z^{\leq 2}(\Phi_1\partial^{3}\Phi_2)\right\| \\
			\overset{\eqref{Poincare ineq.}}&{\lesssim} (1+t)^{-1} \sum_{\tau\in S_2} \left\| \sigma Z^{\leq 2}\partial^{2} \Phi_{\tau(1)} \right\|\cdot \left| \partial Z^{\leq 2} \Phi_{\tau(2)} \right|_\infty + \left| \partial Z^{\leq 2} \Phi_1 \right|_\infty \cdot \left\| \sigma \partial^{3}Z^{\leq 2} \Phi_2 \right\| \\
			&\lesssim(1+t)^{-\frac{1}{2}-\mu}\eta\chi_4.
		\end{aligned}
	\end{equation}
	
	The $\partial\Phi_{1} \partial \Phi_{2}$ terms are bounded by
	\begin{equation} 
		\begin{aligned}
		&\ \ \ \left\|\sigma \partial^{2} Z^{\leq 2} ( \partial \Phi_{1} \partial \Phi_{2}) \right\| \\
		& \lesssim \sum_{\tau\in S_2} \left(  \left\| \sigma Z^{\leq 2} \partial \Phi_{\tau(1)}Z^{\leq 2} \partial^{3} \Phi_{\tau(2)} \right\| + \left\|\sigma Z^{\leq 2} \partial^{2} \Phi_{\tau(1)} Z^{\leq 1} \partial^{2} \Phi_{\tau(2)}  \right\| \right) \\
			& \lesssim \sum_{\tau\in S_2}\left(  \left\| \sigma Z^{\leq 2}\partial^{3} \Phi_{\tau(2)} \right\|\cdot \left| \partial Z^{\leq 2} \Phi_{\tau(1)} \right|_\infty + \left\| \sigma Z^{\leq 2}\partial^{2} \Phi_{\tau(1)} \right\|\cdot \left| \partial Z^{\leq 1} \Phi_{\tau(2)} \right|_\infty\right)  \\
			&\lesssim(1+t)^{-\frac{1}{2}-\mu}\eta\chi_4.
		\end{aligned}
	\end{equation}
	
The $\Phi_1\partial\Phi_{2} \partial \Phi_{3}$ terms are bounded by
\begin{equation}
	\begin{aligned}
		&\ \ \ \left\| \sigma \partial^{2} Z^{\leq 2}( \Phi_1\partial\Phi_{2} \partial \Phi_{3}) \right\| \\
		&\lesssim \sum_{\tau\in S_3} \left( \left\| \sigma Z^{\leq 2}(\partial^{3} \Phi_{\tau(1)}\partial\Phi_{\tau(2)} \Phi_{\tau(3)}) \right\| + \left\| \sigma Z^{\leq 2}(\partial^{2} \Phi_{\tau(1)}\partial^{2}\Phi_{\tau(2)} \Phi_{\tau(3)}) \right\| \right) \\
		&\ \ \ +\sum_{\tau\in S_3}  \left\| \sigma Z^{\leq 2}(\partial^{2} \Phi_{\tau(1)}\partial\Phi_{\tau(2)} \partial \Phi_{\tau(3)}) \right\| \\
		 \overset{\eqref{L^infty for Z4Phi}}&{\lesssim}M^{-1} \sum_{\tau\in S_2} \left(\left\| \sigma Z^{\leq 2} \partial^3 \Phi_{\tau(1)} \right\| \cdot | \partial Z^{\leq 2} \Phi_{\tau(2)} |_{\infty} + \left\| \sigma Z^{\leq 2} \partial^2 \Phi_{\tau(1)} \right\| \cdot | \partial^2 Z^{\leq 1} \Phi_{\tau(2)} |_{\infty}  \right) \\
		&\ \ \  +  \sum_{\tau\in S_3} \left\| \sigma Z^{\leq 2} \partial^2 \Phi_{\tau(1)} \right\| \cdot | \partial Z^{\leq 2} \Phi_{\tau(2)} |_{\infty}| \partial Z^{\leq 2} \Phi_{\tau(3)} |_{\infty}\\
		&\lesssim M^{-1} (1+t)^{-\frac{1}{2}-\mu}\eta\chi_4.
	\end{aligned}
\end{equation}

The $\Phi_1\partial^2\Phi_{2}$ terms are bounded by
\begin{equation} 
	\begin{aligned}
		&\ \ \ \left\| \sigma \partial^{2} Z^{\leq 2}( \Phi_{1} \partial^{2} \Phi_{2}) \right\| \\
		&\lesssim \left\| \sigma Z^{\leq 2} \partial \Phi_{1} Z^{\leq 2} \partial^{3} \Phi_{2} \right\| + \sum_{\tau\in S_2} \left\| \sigma Z^{\leq 2} \partial^{2} \Phi_{\tau(1)}Z^{\leq 1} \partial^{2} \Phi_{\tau(2)} \right\| + \left\| \sigma Z^{\leq 2}\Phi_{1} Z^{\leq 2}\partial^{4} \Phi_{2} \right\|  \\
		& \lesssim \left\| \sigma Z^{\leq 2} \partial^3 \Phi_{2} \right\| \cdot | \partial Z^{\leq 2} \Phi_{1} |_{\infty} + \sum_{\tau\in S_2} \left\| \sigma Z^{\leq 2} \partial^2 \Phi_{\tau(1)} \right\| \cdot | \partial^2 Z^{\leq 1} \Phi_{\tau(2)} |_{\infty} \\
		&\ \ \  + \left\| \sigma^{2} Z^{\leq 2} \partial^{4} \Phi_{2} \right\| \cdot | \sigma^{-1} Z^{\leq 2} \Phi_{1} |_{\infty} \\
		\overset{\eqref{Hardy inequality near boundary}}&{\lesssim} (1+t)^{-\frac{1}{2}-\mu}\eta(\chi_4 + \tilde{\chi}_{5}).
	\end{aligned}
\end{equation}

The $\Phi_{1} \Phi_{2} \partial^{2} \Phi_{3}$ terms are bounded by
\begin{equation} 
	\begin{aligned}
		&\ \ \ \left\| \sigma \partial^{2} Z^{\leq 2}( \Phi_1\Phi_{2} \partial^{2} \Phi_{3}) \right\| \\
		& \lesssim \sum_{\tau\in S_3} \left(  \left\| \sigma Z^{\leq 2}( \Phi_{\tau(1)}\partial^{2}\Phi_{\tau(2)} \partial^{2} \Phi_{\tau(3)}) \right\| + \left\| \sigma Z^{\leq 2}( \partial\Phi_{\tau(1)}\partial\Phi_{\tau(2)} \partial^{2} \Phi_{\tau(3)}) \right\| \right) \\
		&\ \ \  + \sum_{\tau\in S_3} \left( \left\| \sigma Z^{\leq 2}( \Phi_{\tau(1)}\partial\Phi_{\tau(2)} \partial^{3}\Phi_{\tau(3)}) \right\| + \left\| \sigma Z^{\leq 2}( \Phi_{\tau(1)}\Phi_{\tau(2)}\partial^{4} \Phi_{\tau(3)}) \right\| \right) \\
		& \lesssim \sum_{\tau\in S_3} \left\| \sigma Z^{\leq 2} \partial^2 \Phi_{\tau(1)} \cdot Z^{\leq 1} \partial^2  \Phi_{\tau(2)} \right\| \cdot | Z^{\leq 2} \Phi_{\tau(3)} |_{\infty} \\
		&\ \ \  +\sum_{\tau\in S_3} \left\| \sigma Z^{\leq 2} \partial^{\leq 1} \partial^2 \Phi_{\tau(1)} \right\| \cdot |\partial Z^{\leq 2} \Phi_{\tau(2)} |_{\infty} | \partial^{\leq 1}Z^{\leq 2} \Phi_{\tau(3)} |_{\infty}  \\
		&\ \ \  +\sum_{\tau\in S_3} \left\| \sigma Z^{\leq 2} \partial^4 \Phi_{\tau(3)} \right\| \cdot |\sigma^{-1} \partial Z^{\leq 2} \Phi_{\tau(2)} |_{\infty} | Z^{\leq 2} \Phi_{\tau(3)} |_{\infty}  \\
		\overset{\eqref{L^infty for Z4Phi}\eqref{Hardy inequality near boundary}}&{\lesssim} M^{-1} (1+t)^{-\frac{1}{2}-\mu}\eta(\chi_4 + \tilde{\chi}_{5}).
	\end{aligned}
\end{equation}

By combining all the estimates, we obtain the desired result. 
\end{proof}

\section{The estimate of auxiliary energies}
In this section we will show that the auxiliary energies $\eta$, $\chi_{5}$ and $\tilde\chi_{5}$ are actually controlled by $E_{5}$.
 
\begin{lem} \label{thmtildechi}
	If the bootstrap assumptions hold on $[0,T]\subset[0,T_\epsilon]$, then we have
	\begin{equation} 
		\label{estimate of tilde chi}
		\begin{aligned}
			\tilde{\chi}_{5} [\theta,u] (t) \lesssim \chi_{5} [\theta,u] (t)+ E_{3}[\theta,u](t).
		\end{aligned}
	\end{equation}
	holds for all $t\in[0,T]$. 
\end{lem}

\begin{proof}
	Suppose $\Phi\in\{\theta,u\}$. From Lemma \ref{Klainerman-Sideris estimate}, we have
	\begin{equation}
		\begin{aligned}
			\left\| \sigma^2 \partial^{2} Z^{\leq 2} \Phi \right\| \lesssim \left\|\sigma  \partial^3 Z^{\leq 3} \Phi \right\| + (1+t) \left\|\sigma \Box Z^{\leq 2} \Phi \right\| .
		\end{aligned}
	\end{equation}	
	This implies that
	\begin{equation}
		\begin{aligned}
			\tilde\chi_{5} [\Phi] (t) \lesssim \chi_5[\Phi](t)+ (1+t)^{1+ \frac{\mu}{2} } \left\| \sigma \Box \partial^{2}Z^{\leq 2} \Phi \right\|.
		\end{aligned}
	\end{equation}	

	From \eqref{eqn of Z^alpha theta} \eqref{eqn of Z^alpha u}, we obtain
	\begin{equation} 
		\begin{aligned}
		&\left\| \sigma \Box \partial^{2}Z^{\leq 2} \Phi \right\| \\
		=& \left\| \sigma \partial^{2} \Box Z^{\leq 2} \Phi \right\| \\
		 \lesssim& \left\| \sigma  \partial^{2} Z^{\leq 2}[(1+t)^{-1}\partial\Phi ]\right\| +  \left\| \sigma \partial^{2}Z^{\leq 2}[(1+t)^{-2}\Phi]\right\|+\sum_{\substack{|\alpha|\le4\\ \alpha_0+\alpha_1+\alpha_2\ge2}}\left\| \sigma F_{\Phi}^{(\alpha)} \right\|.
		\end{aligned}
	\end{equation}

Since $|Z(1+t)^{-1}|\lesssim(1+t)^{-1}$, $|\partial(1+t)^{-k}|\lesssim(1+t)^{-k-1}$, it is easy to obtain
 \begin{equation} 
		\begin{aligned}
			&\left|\partial^{2} Z^{\leq 2}[(1+t)^{-1}\partial\Phi ]\right|+\left|\partial^{2}Z^{\leq 2}[(1+t)^{-2}\Phi]\right|\\
			\lesssim&(1+t)^{-1} \left|\partial^{2} Z^{\leq 2} \partial^{\leq 1} \Phi  \right| + (1+t)^{-3} \left| Z^{\leq 2} \partial \Phi \right| + (1+t)^{-4} \left| Z^{\leq 2} \Phi \right|.
		\end{aligned}
	\end{equation}
Noting that $ \sigma \leq 1+t $ on $B_{1+t}(0)$, we have
 \begin{equation} 
	\begin{aligned}
		&\left\| \sigma\partial^{2} Z^{\leq 2}[(1+t)^{-1}\partial\Phi ]\right\|+\left\|\sigma\partial^{2}Z^{\leq 2}[(1+t)^{-2}\Phi]\right\| \\
		\lesssim &(1+t)^{-1} \left\|\sigma \partial^{2} Z^{\leq 2} \partial^{\leq 1} \Phi \right\| +  (1+t)^{-2} \left\| Z^{\leq 2} \partial \Phi \right\|+(1+t)^{-3} \left\|Z^{\le2}\Phi\right\|\\
		\lesssim&(1+t)^{-1- \frac{\mu}{2} } \chi_{4} + (1+t)^{-2- \frac{\mu}{2} }E_{2}.
	\end{aligned}
\end{equation}
From Lemma \ref{lem(1+t-x)partial2FPhi}, the forcing terms are bounded by
	\begin{equation} 
		\begin{aligned}
			\sum_{\substack{|\alpha|\le4\\ \alpha_0+\alpha_1+\alpha_2\ge2}}\left\| \sigma F_{\Phi}^{(\alpha)} \right\|&\lesssim\sum_{\substack{|\alpha|\le4\\ \alpha_0+\alpha_1+\alpha_2\ge2}}\left\| \sigma \hat Z^\alpha F_{\Phi} \right\|+\left\| \sigma \left(F_{\Phi}^{(\alpha)}-\hat Z^\alpha F_{\Phi}\right) \right\|\\
			\overset{\eqref{relation between Z and Zhat}\eqref{estimate of F^alpha-Z^alpha F, ineq}}&{\lesssim} \left\| \sigma  \partial^{2}Z^{\leq 2} F_{\Phi} \right\|+(1+t)^{-2}\left\|\sigma\partial Z^{3}\Phi\right\|\\
			 & \lesssim (1+t)^{-1-\frac{\mu}{2}}(1+t)^{ \frac{1-\mu}{2} } \eta\cdot ( \tilde{\chi}_{5} + \chi_{4})+(1+t)^{-1-\frac{\mu}{2}}E_3\\
			\overset{\eqref{usage of T_epsilon}}&{\lesssim}(1+t)^{-1-\frac{\mu}{2}} \left[M^{-1} ( \tilde{\chi}_{5}  + \chi_{4})+E_3\right] .
		\end{aligned}
	\end{equation}
In summary, we have that
	\begin{equation} 
		\begin{aligned}
			(1+t)^{1+ \frac{\mu}{2} } \left\| \sigma \Box \partial^{2}Z^{\leq 2} \Phi \right\| 
			& \lesssim E_{3} + \chi_{4}+M^{-1} \tilde{\chi}_{5} .
		\end{aligned}
	\end{equation}
Consequently, we arrive at
	\begin{equation} 
		\begin{aligned}
			\tilde{\chi}_{5} [\Phi] (t) & \lesssim \chi_{5} [\theta,u] + E_{3}[\theta,u]+M^{-1} \tilde{\chi}_{5} [\theta,u] .
		\end{aligned}
	\end{equation}
Since $ M $ is large enough, we then have 
	\begin{equation} 
		\begin{aligned}
			\tilde{\chi}_{5} [\theta,u] (t) \lesssim \chi_{5} [\theta,u] (t) + E_{3}[\theta,u](t).
		\end{aligned}
	\end{equation}
This completes the estimate of $\tilde\chi$.

\end{proof}

\begin{lem} \label{thmchiE}
	If the bootstrap assumption holds on $[0,T]\subset[0,T_\epsilon]$, then we have
	\begin{equation}
		\label{estimate of chi}
		\chi_5[\theta,u](t)\lesssim E_5[\theta,u](t)+M^{-1}\left(\eta[\theta,u](t)+\tilde\chi_5[\theta,u](t)\right)
	\end{equation}
	holds for all $t\in[0,T]$. 
\end{lem}

\begin{proof}
	Suppose $\Phi\in\{\theta,u\}$. From Lemma \ref{Klainerman-Sideris estimate}, it holds that
	\begin{equation}
		\begin{aligned}
			\left\| \sigma \partial^{2} Z^{\leq 4} \Phi \right\| \lesssim \left\|  \partial Z^{\leq 5} \Phi \right\| + (1+t) \left\| \Box Z^{\leq 4} \Phi \right\|,
		\end{aligned}
	\end{equation}	
	which implies that
	\begin{equation}
		\begin{aligned}
			\chi_{5} [\Phi] (t) \lesssim E_{5} [\Phi] (t)+ (1+t)^{ 1 + \frac{\mu}{2} } \left\| \Box  Z^{\leq 4} \Phi \right\| .
		\end{aligned}
	\end{equation}	
	
	By combining \eqref{eqn of Z^alpha theta} and \eqref{eqn of Z^alpha u}, we have
	\begin{equation} \label{tBoxf}
		\begin{aligned}
			 \left\| \Box Z^{\leq 4} \Phi \right\| &\lesssim (1+t)^{-1} \left\| \partial Z^{\leq 4} \Phi  \right\| + \sum\limits_{ |\alpha| \leq 4} \left\| F_{\Phi}^{(\alpha)} \right\|\\
			 \overset{\eqref{relation between Z and Zhat}}&{\lesssim}(1+t)^{-1-\frac{\mu}{2}}E_4+\underbrace{\|\hat{Z}^{\le4}F_{\Phi,1}\|}_{\eqref{estimate of Z^alpha F1}}+\underbrace{\|\hat{Z}^{\le4}F_{\Phi,2}\|}_{\eqref{estimate of Z^alpha F2}}+\sum\limits_{ |\alpha| \leq 4} \underbrace{\left\| F_{\Phi}^{(\alpha)} -\hat Z^\alpha F_\Phi\right\|}_{\text{Lemma \ref{lemF^alpha-Z^alphaF}}}\\
			 &\lesssim(1+t)^{-1-\frac{\mu}{2}}E_4+(1+t)^{-\frac{1}{2}-\mu}
\left[E_5(E_5+\eta+\chi_5+\tilde\chi_5)+\eta\chi_5\right],
		\end{aligned}
	\end{equation}
which finally gives by invoking the bootstrap assumption and \eqref{usage of T_epsilon} together
\begin{equation}
	\begin{aligned}
		\chi_5[\Phi](t)&\lesssim E_5+M^{-1}(\eta+\chi_5+\tilde{\chi}_5),
	\end{aligned}
\end{equation}
and the result follows since $M$ is large enough. 
\end{proof}

\begin{cor}
	If the bootstrap assumption holds on $[0,T]\subset[0,T_\epsilon]$, then we have
	\begin{equation}
		\label{estimate of chi and tilde chi}
		\chi_5[\theta,u](t)+\tilde\chi_5[\theta,u](t)\lesssim E_5[\theta,u](t)+M^{-1}\eta[\theta,u](t)
	\end{equation}
	holds for all $t\in[0,T]$. 
\end{cor}
\begin{proof}
	Plugging \eqref{estimate of tilde chi} into \eqref{estimate of chi}, we obtain
	\begin{equation}
		\chi_5\lesssim E_5+M^{-1}(\eta+\tilde \chi_5)\lesssim E_5+M^{-1}\eta+M^{-1}\chi_5, 
	\end{equation}
	which yields \eqref{estimate of chi and tilde chi} due to the large constant $M$. The estimate of $\tilde\chi_5$ follows easily from invoking \eqref{estimate of tilde chi} again.  
\end{proof}

The next lemma aims to control $\eta$ by $E$ and $\chi$.
\begin{lem} \label{lemetachi}
	
	If $ \Phi \in C^\infty([0,T]\times\mathbb{R}^2)$, and $\operatorname{supp}\Phi(t,\cdot)\subset \overline{B_{1+t}(0)}$, then it holds that
	\begin{equation}\label{eta1}
		\begin{aligned}
			\eta[\Phi](t) \lesssim E_{4}[\Phi] (t)+\chi_{4}[ \Phi](t).
		\end{aligned}
	\end{equation}
\end{lem}

\begin{proof}
	If $|x|\ge\frac{1+t}{2}$, it follows from Lemma \ref{KSlow} that
	\begin{equation}\label{eta2}
		\begin{aligned}
			(1+t)^{\frac{1+\mu}{2}}|Z^{\leq 2} \partial \Phi (t,x)| \lesssim  (1+t)^{- \frac{\mu}{2}} \left\| Z^{\leq 4} \partial \Phi \right\|\lesssim E_4[\Phi](t).
		\end{aligned}
	\end{equation}
If $ |x| \leq \frac{t+1}{2} $, we have $1+t\lesssim\sigma$. Thus by Lemma \ref{lem1+t-x}, we see that
	\begin{equation} \label{eta3}
		\begin{aligned}
			&\ \ \ (1+t)^{\frac{1+\mu}{2}}|Z^{\leq 2} \partial \Phi (t,x)|\lesssim (1+t)^{\frac{\mu}{2} } |\sigma^{ \frac{1}{2} } Z^{\leq 2} \partial \Phi |_{\infty }\\
			&\lesssim (1+t)^{\frac{\mu}{2} }\left[ \left\| \partial Z^{\leq 2} \Phi \right\|^{ \frac{1}{2} } \left( \left\| \sigma \partial^{3}Z^{\leq 2} \Phi \right\|^{ \frac{1}{2} } + \left\| \partial^{2}Z^{\le2} \Phi \right\|^{ \frac{1}{2} }   \right)+\left\|\sigma^{\frac{1}{2}}\partial^2Z^{\le2}\Phi\right\|\right]\\
			& \lesssim  E_{2}^{\frac{1}{2}} [\Phi] \cdot \chi_{4}^{\frac{1}{2}} [ \Phi] + E_{2}^{\frac{1}{2}} [\Phi] \cdot E_{3}^{\frac{1}{2}} [\Phi] \\
			& \lesssim \chi_{4}[ \Phi]+ E_{3}[\Phi].
		\end{aligned}
	\end{equation}		
Then \eqref{eta1} follows from 	\eqref{eta2} and \eqref{eta3}.
\end{proof}

\begin{cor} \label{thmetaE}
	If the bootstrap assumptions hold on $[0,T]\subset[0,T_\epsilon]$, then we have
	\begin{equation}
		\label{estimate of auxiliary energies}
		\begin{aligned}
			\eta[\theta,u](t)+\chi_5[\theta,u](t)+\tilde\chi_5[\theta,u](t) \lesssim  E_{5} [ \theta ,u]
		\end{aligned}
	\end{equation}
	holds for all $t\in[0,T]$.

\end{cor}
\begin{proof}
	By Lemma \ref{lemetachi}, we find that $\eta\overset{\eqref{estimate of chi and tilde chi}}{\lesssim}E_4+E_5+M^{-1}\eta$. 
	Then the choice of $M$ yields $\eta\lesssim E_5$. The estimate of $\chi_5$ and $\tilde\chi_5$ is obtained by invoking \eqref{estimate of chi and tilde chi} again. 
\end{proof}

\section{The Estimate of $E_{5}[u]$}

\begin{lem} \label{energy estimate of u}
	If the bootstrap assumption holds on $[0,T]\subset[0,T_\epsilon]$, then we have
	\begin{equation} 
		\begin{aligned}
			\sum_{|\alpha|\le k}\|\partial Z^\alpha u(t)\|\lesssim\sum_{|\alpha|\le k}\|\partial Z^\alpha \theta(t)\|
		\end{aligned}
	\end{equation}
	holds for all $t\in[0,T]$, and for $0\le k\le5$. 
\end{lem}

\begin{proof}
	It follows from $ \operatorname{curl} u \equiv 0  $ (see Remark \ref{vorticity free}) that
	\begin{equation} 
		\begin{aligned}
			& \operatorname{curl} Z^{\alpha} u \overset{\eqref{representation of commutators}}{=} Z^{\alpha} \underbrace{\operatorname{curl} u}_{=0} + \partial Z^{ < |\alpha|} u = \partial Z^{ < |\alpha|} u,  \\
			& \operatorname{div} Z^{\alpha} u \overset{\eqref{representation of commutators}}{=}  Z^{\alpha} \operatorname{div} u + \partial Z^{ < |\alpha|} u.
		\end{aligned}
	\end{equation}
	From equations in $ (\ref{2} )$, we have
	\begin{equation}
		\begin{aligned}
			& \nabla\cdot u= - \partial_t\theta - u\cdot\nabla\theta-(\gamma-1) \theta\nabla\cdot u, \\
			& \partial_tu = -\frac{\mu}{1+t} u - u\cdot \nabla u-\nabla\theta.
			\label{Du by Dtheta}
		\end{aligned}
	\end{equation}
	Applying $Z^\alpha$ to both sides derives
	\begin{equation}
		\begin{aligned}
			& Z^{\alpha} \operatorname{div} u  = - Z^{\alpha}\partial_t\theta - Z^{\alpha} (u\cdot\nabla\theta) - (\gamma-1) Z^{\alpha} (\theta \operatorname{div} u), \\
			& Z^{\alpha}\partial_tu = - \mu Z^{\alpha} ( \frac{u}{1+t} ) - Z^{\alpha} ( u\cdot \nabla u ) - Z^{\alpha} \nabla\theta.
		\end{aligned}
		\label{ZDu by ZDtheta}
	\end{equation}
	
We shall prove the lemma by induction. If $k=0$, then we have
	\begin{equation}
		\begin{aligned}
			\label{estimate of nabla u}
			\left\| \nabla u \right\| \overset{\eqref{estimate of Du by div u}}&{\le} \left\| \operatorname{div} u \right\|\overset{\eqref{Du by Dtheta}\eqref{L^infty for Z4Phi}}{\lesssim} \left\| \partial\theta \right\|+M^{-1}\|\partial u\|, \\
		\end{aligned}
	\end{equation}
	\begin{equation}
		\begin{aligned}
			\left\| \partial_{t} u \right\| & \overset{\eqref{Du by Dtheta}}{\lesssim} (1+t)^{-1} \left\| u \right\| + | u |_\infty \cdot \|\nabla u\| + \left\| \nabla\theta  \right\| \\
			\overset{\eqref{L^infty for Z4Phi}\eqref{Poincare ineq.}}&{\le} \left\| \nabla u \right\| + \left\| \nabla \theta \right\|\overset{\eqref{estimate of nabla u}}{\lesssim} \left\| \partial\theta \right\|+M^{-1}\|\partial u\|.  \\
		\end{aligned}
	\end{equation}
	Summing up the above two inequalities and noting that $M$ is large enough, we get 
\[
\|\partial u\|\lesssim\|\partial\theta\|.  
\]
	
	Next we suppose $1\le k\le 5$, and assume that $\|\partial Z^{<k}u\|\lesssim\|\partial Z^{<k}\theta\|$. From \eqref{ZDu by ZDtheta}, for $|\alpha|=k$, we have 
	\begin{equation}
		\begin{aligned}
			&\left\| \nabla Z^{\alpha} u \right\|\\
\overset{\eqref{estimate of Du by div u}}{\le} &\left\| \operatorname{div} Z^{\alpha} u \right\| \overset{\eqref{representation of commutators}}{\le}
			\left\| Z^{\alpha} \operatorname{div} u\right\| +\left\|\partial Z^{<k} u \right\| \\
			\overset{\eqref{ZDu by ZDtheta}}{\lesssim}& \left\| Z^{\alpha} \partial_{t} \theta \right\| + \left\| Z^{\alpha} (u \cdot \nabla \theta) \right\| + \left\| Z^{\alpha} (\theta \operatorname{div} u) \right\| + \left\| \partial Z^{<k} \theta\right\|\\
			\overset{\eqref{representation of commutators}}{\lesssim}& \left\| \partial Z^{\le k} \theta \right\|+ \|Z^{\le k}uZ^{\le2}\partial\theta\|+ \|Z^{\le 2}uZ^{\le k}\partial\theta\|+\left\| Z^{\le k} \theta Z^{\le2} \partial u \right\|\\
&+ \left\| Z^{\le 2} \theta Z^{\le k} \partial u \right\|\\
			\overset{\eqref{Poincare ineq.}}{\lesssim}&\left\| \partial Z^{\le k} \theta \right\|+(1+t)^{\frac{1-\mu}{2}}\eta\left(\left\| \partial Z^{\le k} \theta \right\|+\left\| \partial Z^{\le k}u \right\|\right)\\
			\overset{\eqref{bootstrap assumption}\eqref{usage of T_epsilon}}{\lesssim}&\left\| \partial Z^{\le k} \theta \right\|+M^{-1}\left\| \partial Z^{\le k}u \right\|. 
		\end{aligned}
		\label{estimate of DZu}
	\end{equation}
	
	For the time derivative of $Z^\alpha u$, using the fact that $|Z(1+t)^{-1}|\lesssim(1+t)^{-1}$, we have
	\begin{equation}
		\begin{aligned}
			\left\| \partial_t Z^ {\alpha} u \right\| \overset{\eqref{representation of commutators}}&{\le} \left\| Z^ {\alpha} \partial_t u \right\|+ C \left\| \partial Z^ {<k} u \right\| \\
			\overset{\eqref{ZDu by ZDtheta}}&{\lesssim} \left\| Z^{\alpha} \left( \frac{1}{1+t}  u\right) \right\| + \left\| Z^{\alpha} ( u \cdot \nabla u)  \right\| + \left\| Z^{\alpha} \nabla \theta \right\|   + \left\| \partial Z^ {<k} \theta  \right\| \\
			\overset{\eqref{representation of commutators}}&{\lesssim}(1+t)^{-1}\left\|Z^{\alpha}u   \right\| +  (1+t)^{-1}\left\|Z^ {<k}u\right\|\\
			&\ \ \ \ + \left\| Z^{\leq 2} u  Z^{\leq k} \partial u \right\| + \left\|Z^{\leq 2} \partial u Z^{\leq k} u \right\|+\left\| \partial Z^ {\le k} \theta  \right\| \\
			\overset{\eqref{Poincare ineq.}}&{\lesssim} \left\| \nabla Z^ {\alpha} u  \right\|+ (1+t)^{\frac{1-\mu}{2}}\eta\left\|\partial Z^{\le k}u\right\|+\left\| \partial Z^ {\le k} \theta  \right\|\\
			\overset{\eqref{usage of T_epsilon}\eqref{bootstrap assumption}\eqref{estimate of DZu}}&{\lesssim}\left\| \partial Z^{\le k} \theta \right\|+M^{-1}\left\| \partial Z^{\le k}u \right\|.
		\end{aligned}
		\label{estimate of DtZu}
	\end{equation}
	Summing up \eqref{estimate of DZu} and \eqref{estimate of DtZu} with respect to $|\alpha|=k$, we get 
\[\|\partial Z^{\le k}u\|\lesssim\|\partial Z^{\le k}\theta\|.
\] 
This completes the proof.
\end{proof}

\begin{thm}
	If the bootstrap assumptions hold on $[0,T]\subset[0,T_\epsilon]$, then we have
	\begin{equation} 
		\label{estimate of E_5[u]}
		\begin{aligned}
			E_5[u](t)\lesssim E_5[\theta](t)
		\end{aligned}
	\end{equation}
	holds for all $t\in[0,T]$. 
\end{thm}
\begin{proof}
	This is a direct consequence of Lemma \ref{energy estimate of u} and the definition \eqref{def of Ek} of energy $E_k$. 
\end{proof}

\section{The estimate of $E_{5} [ \theta ]$}
\label{section estimate of E5theta}
In this section, we finally come to a somewhat standard energy estimate. To proceed, we introduce the multiplier
\begin{equation}\label{multi}
	m_{\mu}^{\alpha} \triangleq \frac{\mu}{2} (1+t)^{\mu -1} Z^{\alpha}\theta+  (1+t)^{\mu} \partial_{t} Z^{\alpha} \theta =m_{\mu 1}^\alpha+m_{\mu 2}^\alpha
\end{equation}
for any multi-index $\alpha\in\mathbb{N}^5$. Clearly we have
\begin{equation}
	\label{estimate of the multiplier}
	\|m_{\mu}^\alpha\|\lesssim\|m_{\mu 1}^\alpha\|+\|m_{\mu 2}^\alpha\|\lesssim (1+t)^{\frac{\mu}{2}}E_{|\alpha|}[\theta]. 
\end{equation}

\begin{lem} [Multiplying the multiplier] \label{lemEF}
	If the solution $(\theta,u)$ exists on $[0,t]$, then we have
	\begin{equation} \label{energyoftheta}
		\begin{aligned}
			&  E_{5}^{2} [ \theta ] (t) + \frac{\mu(2-\mu)}{4} G_{5}^{2} [ \theta ] (t)\lesssim E_{5}^{2} [ \theta ] ( 0) +  \sum\limits_{ 0 \leq |\alpha| \leq 5} \left|\int_{0}^{t}\int_{ \R^{2} }  F_{\theta }^{(\alpha)} \cdot m_{\mu}^{\alpha} dxds \right|,
		\end{aligned}
	\end{equation}
	where $G_k^2[\theta]$ is defined by	
	\begin{equation}
		\begin{aligned}
		G_{k}^{2} [ \theta ] (t) = \sum\limits_{ 0 \leq |\alpha| \leq k}  \int_{0}^{t} (1+ \tau )^{\mu-3} \left\|  Z^{\alpha} \theta ( \tau) \right\|^{2}  d \tau .
		\end{aligned}
	\end{equation}
\end{lem}

\begin{proof}
Multiplying \eqref{eqn of Z^alpha theta} by $m_{\beta}^{\alpha}=\beta (1+t)^{\mu -1} Z^{\alpha} \theta+ (1+t)^{\mu} \partial_{t} Z^{\alpha} \theta $ yields
	\begin{equation} \label{energy11}
		\begin{aligned}
			& \partial_{t} \left[ \frac{1}{2} (1+t)^{\mu}| \partial Z^{\alpha} \theta |^{2} + \beta (1+t)^{\mu -1} Z^{\alpha} \theta \partial_{t} Z^{\alpha} \theta + \frac{1}{2}\mu\beta (1+t)^{\mu -2} | Z^{\alpha} \theta |^{2} \right] \\
			& + \left(\frac{1}{2}\mu - \beta\right)(1+t)^{\mu -1} | \partial_{t} Z^{\alpha} \theta |^{2} + \left(-\frac{1}{2}\mu + \beta\right)(1+t)^{\mu -1} | \nabla Z^{\alpha} \theta |^{2} \\
			& + \frac{1}{2}\mu\beta (2-\mu ) (1+t)^{\mu -3} | Z^{\alpha} \theta |^{2} + \beta (1-\mu ) (1+t)^{\mu -2} Z^{\alpha} \theta \partial_{t} Z^{\alpha} \theta  \\
			 =& \operatorname{div}\left(\nabla Z^{\alpha} \theta \cdot m^{\alpha}\right) + F_{\theta }^{(\alpha)} \cdot m_\beta^{\alpha}.
		\end{aligned}
	\end{equation}
Further calculation to the last term in the left hand side yields
	\begin{equation} \label{energy11term}
		\begin{aligned}
			&\ \ \ \ \beta (1-\mu) (1+t)^{\mu -2} Z^{\alpha} \theta \partial_{t} Z^{\alpha} \theta \\
			& =  \beta (1-\mu) (1+t)^{\mu -2} \partial_{t} \left( \frac{1}{2} | Z^{\alpha} \theta |^{2} \right) \\
			& =   \partial_{t} \left[ \beta (1-\mu) (1+t)^{\mu -2}  \frac{1}{2} | Z^{\alpha} \theta |^{2} \right]  - \beta (1-\mu ) (\mu -2) (1+t)^{\mu -3}\frac{1}{2} | Z^{\alpha} \theta |^{2}  .
		\end{aligned}
	\end{equation}
	
	Setting the parameter $\beta= \frac{\mu}{2} $, which means our multiplier will be determined to \eqref{multi}. And substituting $ (\ref{energy11term}) $ into $ (\ref{energy11}) $, we finally get
	\begin{equation} \label{energy1}
		\begin{aligned}
			& \partial_{t}\left [ \frac{1}{2} (1+t)^{\mu}| \partial Z^{\alpha} \theta |^{2} +\frac{\mu}{2} (1+t)^{\mu -1} Z^{\alpha} \theta \partial_{t} Z^{\alpha} \theta + \frac{\mu}{4} (1+t)^{\mu -2} | Z^{\alpha} \theta |^{2} \right] \\
			& + \frac{\mu}{4} (2-\mu ) (1+t)^{\mu -3} | Z^{\alpha} \theta |^{2}  \\
			=&\operatorname{div}(\nabla Z^{\alpha} \theta \cdot m^{\alpha}) + F_{\theta }^{(\alpha)} \cdot m_\mu^{\alpha}.
		\end{aligned}
	\end{equation}	
Since $0< \mu \le 1$, we may rewrite the term in the square bracket of the first line in $ (\ref{energy1}) $ as
	\begin{equation}
		\begin{aligned}
			&  \frac{1}{2} (1+t)^{\mu}| \partial Z^{\alpha} \theta |^{2} +\frac{\mu}{2} (1+t)^{\mu -1} Z^{\alpha} \theta \partial_{t} Z^{\alpha} \theta + \frac{\mu}{4} (1+t)^{\mu -2} | Z^{\alpha} \theta |^{2} \\
			 =& (1+t)^{\mu} \left[ \frac{2-\mu}{4} | \partial_{t} Z^{\alpha} \theta |^{2} +  \frac{1}{2} | \nabla Z^{\alpha} \theta |^{2} \right]  + \frac{\mu}{4} (1+t)^{\mu -2} \left[ (1+t)\partial_{t} Z^{\alpha} \theta + Z^{\alpha} \theta  \right]^{2}, 
		\end{aligned}
	\end{equation}
which is equivalent to	
	\begin{equation}
		(1+t)^{\mu} | \partial Z^{\alpha} \theta |^{2}  + (1+t)^{\mu -2} |Z^{\alpha} \theta |^{2}.
	\end{equation}
	
Then integrating $ (\ref{energy1}) $ over $ [0,t] \times \R^{2} $ yields
	\begin{equation} 
		\begin{aligned}
			& (1+t)^{\mu} \left\|  \partial Z^{\alpha} \theta \right\|  ^{2}  + (1+t)^{\mu -2} \left\|Z^{\alpha} \theta \right\|^{2}  + \frac{\mu(2-\mu)}{4} \int_{0}^{t} (1+s)^{\mu-3} \left\|  Z^{\alpha} \theta (s) \right\|^{2}  ds \\
			 \lesssim &E_{5}^{2} [ \theta ] ( 0) +  \left| \int_{0}^{t}\int_{ \R^{2} } F_{\theta }^{(\alpha)} \cdot m_\mu^{\alpha} dxds \right|.
		\end{aligned}
	\end{equation}
Therefore, we get the desired inequality by summing up the above inequality for $\alpha$ from $0$ to $5$. 
\end{proof}

\begin{lem}
	If the bootstrap assumptions hold on $[0,T] \subset [0,T_{\epsilon}]$, then for $|\alpha|\le5$ and $t\in[0,T]$, we have that
	\begin{equation}
		\label{forcing estimate in energy estimate}
		\begin{aligned}
			&\left|\int_0^t\int_{\R^2}F_\theta^{(\alpha)}m_\mu^\alpha dxds\right|\\
\lesssim&\int_0^t(1+s)^{-2}E_5^2[\theta](s)ds+\int_0^t(1+s)^{-\frac{1+\mu}{2}}
E_5^3[\theta](s)ds\\
&+M^{-1}\left ( E_{5}^{2} [\theta](t) +E_{5}^{2} [\theta](0) \right ). 
		\end{aligned}
	\end{equation}
\end{lem}
\begin{proof}
We decompose $F_\theta^{(\alpha)}$ as
\begin{equation}
	F_\theta^{(\alpha)}=(F_\theta^{(\alpha)}-\hat Z^\alpha F_\theta)+\hat Z^\alpha F_\theta. 
\end{equation}
It follows from Lemma \ref{lemF^alpha-Z^alphaF} that
\begin{equation} 
	\begin{aligned}
		&\ \ \ \ \ \left|\int_{0}^{t}\int_{ \R^{2} }  (F_{\theta }^{(\alpha)} - \hat{Z}^{\alpha } F_{\theta } ) \cdot m_\mu^{\alpha} dxds \right| 
		\lesssim \int_{0}^{t} \left\|  F_{\theta }^{(\alpha)} - \hat{Z}^{\alpha } F_{\theta } \right\|  \cdot \left\| m_\mu^{\alpha}  \right\|  ds  \\
		&\overset{\eqref{estimate of the multiplier}}{\lesssim} \int_{0}^{t} (1+s)^{-2-\frac{\mu}{2}} E_{4} [ \theta ] ( s) \cdot (1+s)^{\frac{\mu}{2}} E_{5} [ \theta ] ( s)   ds \\
& \lesssim \int_{0}^{t} (1+s)^{-2} E_{5}^{2} [ \theta ] ( s) ds.
	\end{aligned}
\end{equation}
Using Lemma \ref{estimate of Z^alpha F1}, we can derive that
\begin{equation} \label{estimate of Z^alpha F_theta1 m^alpha}
	\begin{aligned}
		&\ \ \ \left|\int_{0}^{t}\int_{ \R^{2} } \hat{Z}^{\alpha } F_{\theta,1 } \cdot m_\mu^{\alpha} dxds \right| 
		\lesssim \int_{0}^{t} \left\|  \hat{Z}^{\alpha } F_{\theta,1 } \right\|  \cdot \left\| m_\mu^{\alpha}  \right\|  ds  \\
		\overset{\eqref{estimate of the multiplier}}&{\lesssim} \int_{0}^{t} (1+s)^{-\mu-\frac{1}{2}} \eta  [ \theta,u ]( s) \cdot E_{5} [ \theta,u ] ( s) \cdot (1+s)^{\frac{\mu}{2}} E_{5} [ \theta ] ( s)   ds \\
		\overset{\eqref{estimate of auxiliary energies}\eqref{estimate of E_5[u]}}&{\lesssim} \int_{0}^{t} (1+s)^{-\frac{\mu+1}{2}}  E_{5}^{3} [ \theta] ( s) ds.
	\end{aligned}
\end{equation}

We are left with the term $\left| \int_{0}^{t}\int_{ \R^{2} }  F_{\theta,2}^{(\alpha)} \cdot m_\mu^{\alpha} dxds \right|$. The second order derivative part $ F_{\theta,2 } $ of $F_\theta$  consists of terms in the forms of $ \Phi \partial^{\beta} \theta$ with $ \partial^{\beta} \in \{ \partial_{t} \partial_{x_{i}}, \partial_{x_{i}}^{2} \} $, as well as $ u_{i} u_{j} \partial_{ij}^{2} \theta$. 

\underline{\emph{The quadratic terms}}. We use the following decomposition:

\begin{equation} 
	\begin{aligned}
		\hat{Z}^{\alpha } ( \Phi \partial^{\beta} \theta ) 
		& = \left(  \hat{Z}^{\alpha } ( \Phi \partial^{\beta} \theta ) - \Phi\hat{Z}^{\alpha }\partial^{\beta} \theta \right) 
		+ \left(\Phi\hat{Z}^{\alpha }\partial^{\beta} \theta - \Phi\partial^{\beta}\hat{Z}^{\alpha }\theta \right)  \\
		& \ \ \ +  \Phi\partial^{\beta} (\hat{Z}^{\alpha } - {Z}^{\alpha })\theta 
		+   \Phi\partial^{\beta} {Z}^{\alpha }\theta \\
		\overset{\eqref{representation of commutators}}&{ =} Z^{\le2}\Phi Z^{\le4}\partial^{\beta}\theta +   Z^{\le5}\Phi Z^{\le2}\partial^{\beta}\theta +  \Phi\partial^{\beta} {Z}^{\alpha }\theta.
	\end{aligned}
\end{equation}
The lower order terms can be bounded by
\begin{equation} 
	\begin{aligned}
		&\ \ \ \ \left|\int_{0}^{t}\int_{ \R^{2} } \left( Z^{\le2}\Phi Z^{\le4}\partial^{\beta}\theta +   Z^{\le5}\Phi Z^{\le2}\partial^{\beta}\theta\right)  \cdot m_\mu^{\alpha} dxds \right| \\
		& \lesssim \int_{0}^{t} \left( \|Z^{\le2}\Phi_1Z^{\le4}\partial^2\Phi_2\| + \left\| Z^{\leq 5} \Phi_{1} Z^{\leq 2} \partial^{2}  \Phi_{2}  \right\|\right)  \cdot \left\| m_\mu^{\alpha}  \right\| ds  \\ 
		\overset{\eqref{estimate of the multiplier}\eqref{Z^leq2 Z^leq4 partial^2}\eqref{Z^leq5 Z^leq2 partial^2}}&{\lesssim}  \int_{0}^{t} (1+s)^{ -\frac{1+\mu}{2} } E_5\left[ E_5\left( E_5+\chi_5+\tilde{\chi}_5 \right) + \eta \chi_5 \right] ds\\
		\overset{\eqref{estimate of auxiliary energies}\eqref{estimate of E_5[u]}}&{\lesssim}\int_{0}^{t} (1+s)^{-\frac{\mu+1}{2}}  E_{5}^{3} [ \theta] ( s) ds.
	\end{aligned}
\end{equation}
For the highest order term, we have
\begin{equation} 
	\begin{aligned}
		&\left|\int_{0}^{t}\int_{ \R^{2} }   \Phi\partial^{\beta} {Z}^{\alpha }\theta \cdot m_{\mu}^{\alpha} dxds \right
|\\
\lesssim &\left|\int_{0}^{t}\int_{ \R^{2} }   \Phi\partial^{\beta} {Z}^{\alpha }\theta \cdot m^{\alpha}_{\mu 1} dxds \right|  + \left|\int_{0}^{t}\int_{ \R^{2} }   \Phi\partial^{\beta} {Z}^{\alpha }\theta \cdot m^{\alpha}_{\mu 2} dxds \right| \\
		 \triangleq& I_{1} +I_{2}.
	\end{aligned}
\end{equation}
Writing $ \partial^{\beta}=\partial_{i} \partial $ and using integration by parts, we have	
\begin{equation} 
	\begin{aligned}
		I_{1}	& \lesssim \left| \sum\limits_{ |\alpha| \leq 5} \int_{0}^{t}\int_{ \R^{2} }
		\partial_{i} \Phi \partial {Z}^{\alpha }\theta \cdot m^{\alpha}_{\mu 1} dxds\right| 
		+ \left| \sum\limits_{ |\alpha| \leq 5} \int_{0}^{t}\int_{ \R^{2} }
		\Phi \partial {Z}^{\alpha }\theta \cdot \partial_{i} m^{\alpha}_{\mu 1} dxds\right| \\
		& \lesssim \int_{0}^{t} \left(  \left\| \partial \Phi \cdot \partial Z^{\leq 5} \theta \right\| \cdot \left\|m^{\alpha}_{\mu 1} \right\| + \left\| \Phi \cdot \partial Z^{\leq 5} \theta \right\| \cdot \left\|\nabla m^{\alpha}_{\mu 1} \right\| \right) ds \\
		\overset{\eqref{estimate of the multiplier}}&{\lesssim}  \int_{0}^{t} (1+s)^{-\frac{1+\mu}{2}}  \eta  [ \theta,u ]( s) \cdot E_{5}^{2} [ \theta,u ] ( s) ds \\
		\overset{\eqref{estimate of auxiliary energies}\eqref{estimate of E_5[u]}}&{\lesssim} \int_{0}^{t} (1+s)^{-\frac{1+\mu}{2}}  E_{5}^{3} [ \theta] ( s) ds.
	\end{aligned}
\end{equation}	
To control $I_{2}$, we need to consider $\partial^{\beta}=\partial_{t} \partial_{i}  $ and $ \partial^{\beta}=\partial_{i}^{2}  $ separately.
If $\partial^{\beta}=\partial_{t} \partial_{i}  $, it follows from integration by parts that
\begin{equation} 
	\begin{aligned}
		I_{2}
		& = \left|\int_{0}^{t} (1+s)^{\mu} \int_{ \R^{2} }   \Phi\partial_{s} \partial_{i} {Z}^{\alpha }\theta \cdot \partial_{s} {Z}^{\alpha }\theta dxds \right| \\
		& \lesssim \left|\int_{0}^{t} (1+s)^{\mu} \int_{ \R^{2} }   \Phi \partial_{i} (\left| \partial_{s} {Z}^{\alpha }\theta \right|^{2} )dxds \right| 
		\lesssim \left|\int_{0}^{t} (1+s)^{\mu} \int_{ \R^{2} } \partial_{i} \Phi \cdot \left| \partial_{s} {Z}^{\alpha }\theta \right|^{2} dxds \right| \\
		& \lesssim \int_{0}^{t} (1+s)^{\mu} |\partial \Phi|_{\infty} \cdot \left\| \partial Z^{\leq 5} \theta \right\| ^{2}  ds
		\lesssim \int_{0}^{t} (1+s)^{-\frac{1+\mu}{2}}  \eta  [ \theta,u ]( s) \cdot E_{5}^{2} [ \theta,u ] ( s) ds\\
		\overset{\eqref{estimate of auxiliary energies}\eqref{estimate of E_5[u]}}&{\lesssim} \int_{0}^{t} (1+s)^{-\frac{1+\mu}{2}}  E_{5}^{3} [ \theta] ( s) ds.
	\end{aligned}
\end{equation}			
If $\partial^{\beta}=\partial_{i}^{2}$, using again integration by parts, we have that
\begin{equation} 
	\begin{aligned}
		I_{2}
		& = \left|\int_{0}^{t} (1+s)^{\mu} \int_{ \R^{2} }   \Phi\partial_{i}^{2} {Z}^{\alpha }\theta \cdot \partial_{s} {Z}^{\alpha }\theta dxds \right| \\
		& \lesssim \left|\int_{0}^{t} (1+s)^{\mu} \int_{ \R^{2} }   \partial_{i} \Phi \cdot \partial_{i} {Z}^{\alpha } \Phi \cdot \partial_{s} {Z}^{\alpha }\theta  dxds \right| \\
		&+ \left|\int_{0}^{t} (1+s)^{\mu} \int_{ \R^{2} }  \Phi \partial_{s} (\left| \partial_{i} {Z}^{\alpha }\theta \right|^{2}) dxds \right| \\
		& \lesssim \int_{0}^{t} (1+s)^{\mu} |\partial \Phi| \cdot \left\| \partial Z^{\leq 5} \theta \right\| ^{2} 
		+  \left|\int_{0}^{t} (1+s)^{\mu-1} \int_{ \R^{2} }  \Phi \left| \partial_{i} {Z}^{\alpha }\theta \right|^{2} dxds \right| \\
		& +  \left|\left(   \int_{ \R^{2} } (1+t)^{\mu} \Phi(t) \left| \partial_{i} {Z}^{\alpha }\theta(t) \right|^{2} dx
		-  \int_{ \R^{2} }  \Phi(0) \left| \partial_{i} {Z}^{\alpha }\theta(0) \right|^{2} dx \right) \right|  \\
		\overset{\eqref{Poincare ineq.}}&{\lesssim} \int_{0}^{t} (1+s)^{-\frac{\mu+1}{2}}  \eta  [ \theta,u ]( s) \cdot E_{5}^{2} [ \theta,u ] ( s) ds +  M^{-1}\left ( E_{5}^{2} [\theta](t) +E_{5}^{2} [\theta](0)  \right)\\
		\overset{\eqref{estimate of auxiliary energies}\eqref{estimate of E_5[u]}}&{\lesssim} \int_{0}^{t} (1+s)^{-\frac{1+\mu}{2}}  E_{5}^{3} [ \theta] ( s) ds+M^{-1}\left ( E_{5}^{2} [\theta](t) +E_{5}^{2} [\theta](0) \right ).
	\end{aligned}
\end{equation}			

\underline{\emph{The cubic term}}. In the following, we handle the cubic nonlinear term $ u_{i} u_{j} \partial_{ij}^{2} \theta$. We shall use the decomposition:
\begin{equation} 
	\begin{aligned}
		\hat{Z}^{\alpha } ( u_{i} u_{j} \partial_{ij}^{2} \theta ) 
		& = \left(  \hat{Z}^{\alpha } ( u_{i} u_{j} \partial_{ij}^{2} \theta ) -  u_{i} u_{j}\hat{Z}^{\alpha }\partial^{\beta} \theta \right) 
		+ \left( u_{i} u_{j}\hat{Z}^{\alpha }\partial_{ij}^{2} \theta -  u_{i} u_{j}\partial_{ij}^{2}\hat{Z}^{\alpha }\theta \right)  \\
		& \ \ \ +   u_{i} u_{j}\partial_{ij}^{2} (\hat{Z}^{\alpha } - {Z}^{\alpha })\theta 
		+    u_{i} u_{j}\partial_{ij}^{2} {Z}^{\alpha }\theta \\
		\overset{\eqref{representation of commutators}}&{ =} Z^{\le2}(u_{i} u_{j}) Z^{\le4}\partial_{ij}^{2}\theta +   Z^{\le5}(u_{i} u_{j}) Z^{\le2}\partial_{ij}^{2}\theta +  u_{i} u_{j}\partial_{ij}^{2} {Z}^{\alpha }\theta.
	\end{aligned}
\end{equation}	
By \eqref{L^infty for Z4Phi}, we get the following estimates for the lower order terms:
\begin{equation} 
	\begin{aligned}	
		\left\|Z^{\le2}(u_{i} u_{j}) Z^{\le4}\partial_{ij}^{2}\theta \right\| 
		\lesssim \left\|Z^{\le2}u_{i}  Z^{\le4}\partial_{ij}^{2}\theta \right\| \cdot \left|  Z^{\le2}u_{j} \right| _{\infty} 
		\overset{\eqref{Z^leq2 Z^leq4 partial^2}}{\lesssim} 
		M^{-1} (1+t)^{-\frac{1}{2}-\mu} \eta \chi_5,
	\end{aligned}
\end{equation}		
\begin{equation} 
	\begin{aligned}	
		\left\|Z^{\le5}(u_{i} u_{j}) Z^{\le2}\partial_{ij}^{2}\theta \right\| 
		&\lesssim \left\|Z^{\le2}u_{i}  Z^{\le3}u_{j}  Z^{\le2}\partial_{ij}^{2}\theta \right\|\\
		&\lesssim \left\|Z^{\le2}u_{i}  Z^{\le2}\partial_{ij}^{2}\theta \right\| \cdot \left|  Z^{\le3}u_{j} \right| _{\infty} \\
		&\overset{\eqref{Z^leq2 Z^leq4 partial^2}}{\lesssim} 
		M^{-1} (1+t)^{-\frac{1}{2}-\mu} \eta \chi_5.
	\end{aligned}
\end{equation}		
Hence, the lower order terms can be bounded by
\begin{equation} 
	\begin{aligned}
		& \ \ \ \left|\int_{0}^{t}\int_{ \R^{2} } \left(Z^{\le2}(u_{i} u_{j}) Z^{\le4}\partial_{ij}^{2}\theta +   Z^{\le5}(u_{i} u_{j}) Z^{\le2}\partial_{ij}^{2}\theta\right)  \cdot m_{\mu}^{\alpha} dxds \right| \\
		& \lesssim \int_{0}^{t} \left( \|Z^{\le2}(u_{i} u_{j}) Z^{\le4}\partial_{ij}^{2}\theta\| + \left\| Z^{\le5}(u_{i} u_{j}) Z^{\le2}\partial_{ij}^{2}\theta \right\|\right)  \cdot \left\| m_{\mu}^{\alpha}  \right\| ds  \\ 
		\overset{\eqref{L^infty for Z4Phi}\eqref{estimate of the multiplier}}&{\lesssim} M^{-1} \int_{0}^{t} (1+s)^{ -\frac{1+\mu}{2} } E_5\eta \chi_5 ds\overset{\eqref{estimate of auxiliary energies}\eqref{estimate of E_5[u]}}{\lesssim} \int_{0}^{t} (1+s)^{-\frac{1+\mu}{2}}  E_{5}^{3} [ \theta] ( s) ds.
	\end{aligned}
\end{equation}	
To bound the highest order term, we split it as	
\begin{equation} 
	\begin{aligned}
		&\ \ \ \left|\int_{0}^{t}\int_{ \R^{2} }   u_{i} u_{j} \partial_{ij}^{2}{Z}^{\alpha } \theta \cdot m_{\mu}^{\alpha} dxds \right| \\
		& \lesssim \left|\int_{0}^{t}\int_{ \R^{2} }   u_{i} u_{j} \partial_{ij}^{2}{Z}^{\alpha } \theta \cdot m^{\alpha}_{\mu 1} dxds \right|  + \left|\int_{0}^{t}\int_{ \R^{2} }   u_{i} u_{j} \partial_{ij}^{2}{Z}^{\alpha } \theta \cdot m^{\alpha}_{\mu 2} dxds \right| \\
		& \triangleq J_{1} +J_{2}.
	\end{aligned}
\end{equation}			
Using integration by parts, $J_1$ can be estimated as
\begin{equation} 
	\begin{aligned}	
		J_{1} 	& \lesssim \left| \int_{0}^{t}\int_{ \R^{2} } \left(  \partial_{i} (u_{i} u_{j}) \partial_{j} {Z}^{\alpha } \theta \cdot m^{\alpha}_{\mu 1} + u_{i} u_{j} \cdot \partial_{j} {Z}^{\alpha } \theta \cdot \partial_{i} m^{\alpha}_{\mu 1} \right)  dxds \right| \\
		& \lesssim 	\int_{0}^{t} \left| \partial_{i} (u_{i} u_{j}) \right| _{\infty} \left\| \partial_{j} {Z}^{\alpha} \theta  \right\| \cdot \left\|  m^{\alpha}_{\mu 1} \right\| ds
		+ \int_{0}^{t} \left| u_{i} u_{j} \right| _{\infty} \left\| \partial_{j} {Z}^{\alpha} \theta  \right\| \cdot \left\| \nabla m^{\alpha}_{\mu 1} \right\| ds \\
		\overset{\eqref{L^infty for Z4Phi}\eqref{estimate of the multiplier}}&{\lesssim}M^{-1}\int_{0}^{t} (1+s)^{-\frac{1+\mu}{2}}  \eta  [ \theta,u ]( s) \cdot E_{5}^{2} [ \theta,u ] ( s) ds\\
		\overset{\eqref{estimate of auxiliary energies}\eqref{estimate of E_5[u]}}&{\lesssim} \int_{0}^{t} (1+s)^{-\frac{1+\mu}{2}}  E_{5}^{3} [ \theta] ( s) ds.
	\end{aligned}
\end{equation}		
For the term $J_2$, one has
\begin{equation} 
	\begin{aligned}	
		J_{2} & = \left|\int_{0}^{t} (1+s)^{\mu} \int_{ \R^{2} }   u_{i} u_{j} \partial_{ij}^{2} Z^{\alpha } \theta \cdot \partial_{s} Z^{\alpha } \theta dxds \right| \\
		& \lesssim \left|\int_{0}^{t}(1+s)^{\mu}  \int_{ \R^{2} } \left[   \partial_{j} (u_{i} u_{j}) \partial_{i} {Z}^{\alpha } \theta \cdot \partial_{s} Z^{\alpha } \theta + u_{i} u_{j} \cdot \partial_{i} {Z}^{\alpha } \theta \cdot \partial_{s}\partial_{j} Z^{\alpha } \theta \right]   dxds \right|,
	\end{aligned}
\end{equation}			
the first term of which can be controlled by		
\begin{equation} 
	\begin{aligned}		
		&\left| \int_{0}^{t}(1+s)^{\mu}  \int_{ \R^{2} }   \partial_{j} (u_{i} u_{j}) \partial_{i} {Z}^{\alpha } \theta \cdot \partial_{s} Z^{\alpha } \theta  dxds \right|\\
	\overset{\eqref{L^infty for Z4Phi}}{\lesssim} &M^{-1}\int_{0}^{t} (1+s)^{-\frac{\mu+1}{2}}  \eta   E^{2}_5ds\\
		\overset{\eqref{estimate of auxiliary energies}\eqref{estimate of E_5[u]}}{\lesssim}& \int_{0}^{t} (1+s)^{-\frac{1+\mu}{2}}  E_{5}^{3} [ \theta] ( s) ds,
	\end{aligned}
\end{equation}	
while the second term can be estimated as
\begin{equation} 
	\begin{aligned}		
		& \ \ \ \left|\int_{0}^{t}(1+s)^{\mu}  \int_{ \R^{2} }   u_{i} u_{j} \cdot \partial_{i} {Z}^{\alpha } \theta \cdot \partial_{s}\partial_{j} Z^{\alpha } \theta  dxds \right|	\\
		& \leq \frac{1}{2}  (1+t)^{\mu} \left| \int_{ \R^{2} }   u_{i}(t) u_{j}(t) \cdot \partial_{i} {Z}^{\alpha } \theta(t) \cdot \partial_{j} Z^{\alpha } \theta(t)  dx \right|\\
		&\ \ \ +\frac{1}{2}\left| \int_{ \R^{2} }   u_{i}(0) u_{j}(0) \cdot \partial_{i} {Z}^{\alpha } \theta(0) \cdot \partial_{j} Z^{\alpha } \theta(0)  dx \right|\\
		&\ \ \  + \frac{1}{2}\int_{0}^{t} \int_{ \R^{2} } \Big[  \left| \mu(1+s)^{\mu-1}  u_{i} u_{j} \cdot \partial_{i} {Z}^{\alpha } \theta \cdot \partial_{j} Z^{\alpha } \theta \right|\\
		&\ \ \ + \left| (1+s)^{\mu} \partial_{s} (u_{i} u_{j}) \cdot \partial_{i} {Z}^{\alpha } \theta \cdot \partial_{j} Z^{\alpha } \theta \right| 		\Big] dxds \\
		\overset{\eqref{L^infty for Z4Phi}}&{\lesssim} M^{-2}\left ( E_{5}^{2} [\theta](t) +E_{5}^{2} [\theta](0) \right ) +  M^{-1} \int_{0}^{t} (1+s)^{-\frac{1+\mu}{2}}  \eta  [ \theta,u ]( s) \cdot E_{5}^{2} [ \theta,u ] ( s) ds\\
		\overset{\eqref{estimate of auxiliary energies}\eqref{estimate of E_5[u]}}&{\lesssim} M^{-2} \left( E_{5}^{2} [\theta](t) +E_{5}^{2} [\theta](0)  \right)+M^{-1}\int_{0}^{t} (1+s)^{-\frac{1+\mu}{2}}  E_{5}^{3} [ \theta] ( s) ds.
	\end{aligned}
\end{equation}	
Combining all the estimates above, we obtain the desired estimate \eqref{forcing estimate in energy estimate}. 
\end{proof}

\begin{lem} 
	If the bootstrap assumptions hold on $[0,T] \subset [0,T_{\epsilon}]$, then we have
	\begin{equation} 
		\label{estimate of E5 theta}
		\begin{aligned}
			E_{5} [ \theta ] (t)\lesssim E_{5} [ \theta ] ( 0) 
		\end{aligned}
	\end{equation}
	holds for all $t \in [0,T]$.
\end{lem}
\begin{proof}
	Substituting \eqref{forcing estimate in energy estimate} into lemma \ref{lemEF} and noting that $M$ is large enough, we have
	\begin{equation} 
		\begin{aligned}
			E_{5}^{2} [ \theta ] (t) 
			& \lesssim  E_{5}^{2} [ \theta ] ( 0) + \int_{0}^{t} (1+s)^{-2} E_{5}^{2} [ \theta ] ( s) ds+\int_{0}^{t} (1+s)^{-\frac{\mu+1}{2}}  E_{5}^{3} [ \theta] ( s) ds.
		\end{aligned}
	\end{equation}
Employing Gronwall inequality and using the bootstrap assumption yield
	\begin{equation} 
		\begin{aligned}
			E_{5}^{2} [ \theta ] (t) &\lesssim E_{5}^{2} [ \theta ] ( 0)  \exp \left( \int_{0}^{t} (1+s)^{-2}ds +\int_{0}^{t}(1+s)^{ - \frac{1+\mu}{2} }  M \epsilon  ds \right)\\
			&\lesssim E_{5}^{2} [ \theta ] ( 0)  \exp \left( \int_{0}^{t}(1+s)^{ - \frac{1+\mu}{2} }  M \epsilon  ds \right).
		\end{aligned}
	\end{equation}
	If $\mu\in(0,1)$, we obtain
	\begin{equation} 
		\begin{aligned}
			E_{5}^{2} [ \theta ] (t) & \lesssim E_{5}^{2} [ \theta ] ( 0)  \exp \left( (1+T_{\epsilon})^{ - \frac{1-\mu}{2}} M \epsilon \right) \\
			\overset{\eqref{usage of T_epsilon}}&{\lesssim} E_{5}^{2} [ \theta ] ( 0)  \exp (M^{-1}) \lesssim E_{5}^{2} [ \theta ] ( 0),
		\end{aligned}
	\end{equation}
	while if $ \mu = 1 $, recall the choice of $T_{\epsilon}$ in this case:
	\begin{equation} 
		\begin{aligned}
			\ln(1+T_{\epsilon})\epsilon = \frac{1}{M},
		\end{aligned}
	\end{equation}
	 we have
	\begin{equation} 
		\begin{aligned}
			E_{5}^{2} [ \theta ] (t)
			& \lesssim E_{5}^{2} [ \theta ] ( 0)  \exp \left(\ln(1+T_{\epsilon}) M \epsilon\right)\lesssim E_{5}^{2} [ \theta ] ( 0) .
		\end{aligned}
	\end{equation}
\end{proof}

Finally we are able to prove the key bootstrap lemma. 
\begin{proof}[Proof of the bootstrap lemma \ref{bootstrap lemma}]
	From \eqref{estimate of auxiliary energies}\eqref{estimate of E_5[u]}\eqref{estimate of E5 theta}, we have that
	\begin{equation}
		\eta[\theta,u](t)+E_5[\theta,u](t)\le CE_5[\theta](0)\le\frac{1}{2}M\epsilon. 
	\end{equation}
\end{proof}

\section{Acknowledgement}

The authors would like to express their sincere thank to professor Yi Zhou for the helpful discussion.

The second author was partially supported by NSFC(12271487, 12171097).

\appendix
\section{Useful lemmas}
\begin{lem} \label{lem1}
	For any vector-valued function $U(x) = (U_{1}(x), U_{2}(x))\in H^1(\R^2)^2$, there holds that
	\begin{equation} 
		\label{estimate of Du by div u}
		\left\| \nabla U \right\| \leq \left\| \operatorname{curl} U \right\|+\left\| \operatorname{div} U \right\|.
	\end{equation}
\end{lem}
\begin{proof}
	The result follows from the Parseval identity and the fact that $|\xi|^2(|\hat U_1|^2+|\hat U_2|^2)=|\xi_2\hat U_1-\xi_1\hat U_2|^2+|\xi_1\hat U_1+\xi_2\hat U_2|^2$.
\end{proof}

\begin{lem}[Poincar\'e inequality] \label{lem23}
	Suppose $p\in[1,\infty]$. For any $f\in W^{1,p}_0(B_R(0))$, where $B_R(0)=\{x\in\R^2:|x|<R\}$, we have that
	\begin{equation}
		\begin{aligned}
			\|f\|_{L^p(\R^2)}\le R\|\nabla f\|_{L^p(\R^2)}.
		\end{aligned}
		\label{Poincare ineq.}
	\end{equation}	
\end{lem}

\begin{proof}
	It suffices to prove the inequality for $f\in C_c^\infty(B_R(0))$. We first deal with the case $p\in[1,\infty)$. In the polar coordinate $ (r, \phi) $ of $\R^2$, we have
	\begin{equation}
		\begin{aligned}
			|f(r, \phi)|^p & = \left|\int_{r}^{R} \partial_rf (\xi, \phi) d\xi\right|^p \overset{\text{H\"older}}{\le} (R-r)^{p-1} \int_r^R|\partial_rf(\rho,\phi)|^pd\rho.
		\end{aligned}
	\end{equation}
	Hence, the $L^p$ norm of $f$ can be estimated as
	\begin{equation}
		\begin{aligned}
			\int_{ \R^{2} }|f|^pdx&=\int_{0}^{2\pi}d\phi\int_0^R|f(r,\phi)|^prdr\\
			&\le\int_{0}^{2\pi}d\phi\int_0^R\left((R-r)^{p-1}\int_r^R|\partial_rf(\rho,\phi)|^pd\rho\right)rdr\\
			\overset{\text{Fubini}}&{\le}\int_{0}^{2\pi}d\phi\int_0^R\left(\int_0^\rho(R-r)^{p-1}rdr\right)|\partial_rf(\rho,\phi)|^pd\rho\\
			&\le\frac{R^p}{p}\int_{0}^{2\pi}d\phi\int_0^R|\partial_rf(\rho,\phi)|^p\rho d\rho=\frac{R^p}{p}\int_{ \R^{2} }|\partial_rf|^pdx.
		\end{aligned}
	\end{equation}
	Note that $|\partial_rf|\le|\nabla f|$, this completes the proof. The case $p=\infty$ can be proved with the fundamental theorem of calculus. 
\end{proof}	

\begin{lem}[Hardy inequality, \cite{Hardy}, Section 2.1]
	Suppose $n\ge2$, $p\in[1,n)$, and $f\in W^{1,p}(\mathbb{R}^n)$, then we have that
	\begin{equation}
		\label{Hardy inequality}
		\left\|\frac{f}{|x|}\right\|_{L^p(\R^n)}\le\frac{p}{n-p}\|\nabla f\|_{L^p(\R^n)}. 
	\end{equation}
\end{lem}

\begin{lem}[Hardy type inequality in a ball]
	 For smooth $f(t, x)$ on $[0, T]\times\R^2$ with $\operatorname{supp}_xf\subset B\left(0,\frac12+t\right)$, we have that
	\begin{equation}
\begin{aligned}
		\label{Hardy inequality near boundary}
		&\left\|\frac{f}{t+1-|x|}\right\|_{L^\infty(\R^2)}\lesssim\|\nabla f\|_{L^\infty(\R^2)},\\
&\left\|\frac{f}{t+1-|x|}\right\|_{L^2(\R^2)}\lesssim\|\nabla f\|_{L^2(\R^2)}.\\
\end{aligned}
	\end{equation}
\end{lem}
\begin{proof}
By the fundamental theorem of calculus, it is easy to get
\begin{equation}
\begin{aligned}
|f(t, r\omega)|\lesssim&\left|\int_{\mathbb{S}^1}\int_r^{t+1}\partial_rf(t, r\omega)drd\omega\right|\\
&\lesssim |\nabla f|(t+1-r),
\end{aligned}
	\end{equation}	
which yields the first inequality in \eqref{Hardy inequality near boundary}.

For the second inequality, noting that
\begin{equation}
\begin{aligned}
\left\|\frac{f}{t+1-r}\right\|^2_{L^2(\R^2)}\lesssim&\int_0^{t+1}\frac{f^2(t, r\omega)r}{(t+1-r)^2}dr\\
=&-\int_0^{t+1}\frac{\partial_r(f^2(t, r\omega)r)}{t+1-r}dr\\
=&-2\int_0^{t+1}\frac{f\partial_rfr}{t+1-r}dr-\int_0^{t+1}\frac{f^2}{t+1-r}dr\\
\le& -2\int_0^{t+1}\frac{f\partial_rfr}{t+1-r}dr\\
\lesssim& \left\|\frac{f}{t+1-r}\right\|_{L^2(\R^2)}\|\nabla f\|_{L^2(\R^2)},
\end{aligned}
	\end{equation}	
then it follows.

\end{proof}

\begin{lem}[Gagliardo-Nirenberg inequality, \cite{Nir59}, Page 125]
	Let $q \in [1,\infty]$ be a positive extended real quantity. Let $j,m $ be non-negative integers such that $j<m $. Furthermore, let $r \in [1,\infty]$ be a positive extended real quantity, $p\geq 1$ be real and $ \alpha \in [0,1] $ such that the relations
		\begin{equation}
		\frac{1}{p} = \frac{j}{n} + \alpha \left( \frac{1}{r} - \frac{m}{n} \right) + \frac{1-\alpha}{q}, \ \frac{j}{m} \leq \alpha \leq 1
	\end{equation}
hold. Then
	\begin{equation}
		\left\| D^{j} f  \right\|_{L^p(\mathbb{R}^n)}\le C \left\| D^{m} f  \right\|_{L^r(\R^n)} ^{\alpha} \left\| f  \right\|_{L^q(\R^n)} ^{1-\alpha}
	\end{equation}
for any $ f \in L^q(\R^n) $ such that $ D^{m} f \in L^r(\R^n) $, with two exceptional cases:

1. If $j=0$ (with the understanding that $D^{0} f = f $ ), $q=+\infty$ and $rm<n$, then an additional assumption is needed: either $f$ tends to $0$ at infinity, or $ f \in L^s(\R^n) $ for some finite value of $s$;

2. If $r>1$ and $m-j-\frac{n}{r}$ is a non-negative integer, then the additional assumption $\frac{j}{m} \leq \alpha < 1$ (notice the strict inequality) is needed.

In any case, the constant $C>0$ depends on the parameters $j,m,n,q,r,\alpha$, but not on $f$.

\end{lem}

\begin{rem}
	This is the original version of Gagliardo-Nirenberg interpolation inequality, for functions defined on $\R^n$. In this paper, we apply the result of case $ j=0, m = 2, n=2, p = \infty, q =r=2  $. For the readers' convenience, we provide a brief proof of this case.
\end{rem}

\begin{lem}[Gagliardo-Nirenberg inequality used in this paper] It holds that

	\begin{equation}
	\label{GN inequality}
	\left\|  f  \right\|_{L^\infty (\R^2)}\lesssim \left\| f  \right\|_{L^2(\R^2)} ^{\frac{1}{2}} \left\| \nabla^{2} f  \right\|_{L^2(\R^2)} ^{\frac{1}{2}} 
\end{equation}
for any $ f \in L^2(\R^2) $ such that $ \nabla^{2} f \in L^2(\R^2) $.

\end{lem}

\begin{proof}
	The case $\| f \| = 0$ holds true since $ |f|_\infty= 0 $. Otherwise, we have
	
	\begin{equation}
		\begin{aligned}
			|f|_\infty & \leq \left\| \hat{f} \right\| _{L^{1} (\R^{2})} = \int_{ \R^{2} } | \hat{f} (\xi) | d \xi \\
			& \leq \int_{ |\xi| \leq R } | \hat{f} (\xi) | d \xi + \int_{ |\xi| \geq R } |\xi|^{2} \cdot | \hat{f} (\xi) | \cdot  |\xi|^{-2}   d \xi \\
			& \leq \left(\int_{ |\xi| \leq R }  | \hat{f} (\xi) | ^{2} d \xi \right) ^{\frac{1}{2}} \cdot | B_{R} |^{\frac{1}{2}} \\
			& \ \ \ \ + \left( \int_{ |\xi| \geq R } |\xi|^{4} \cdot | \hat{f} (\xi) | ^{2} d \xi \right) ^{\frac{1}{2}} \cdot \left( \int_{ |\xi| \geq R } |\xi|^{-4} d \xi \right) ^{\frac{1}{2}} \\
			& \lesssim \| f \| \cdot R + \| \nabla^{2} f \| \cdot \left( \int_{R }^{+ \infty} r^{-4} r d r \right) ^{\frac{1}{2}} \\
			&\lesssim \| f \| \cdot R + \| \nabla^{2} f \| \cdot R^{-1} .
		\end{aligned}
	\end{equation}
Taking $ R = \frac{\| \nabla^{2} f \|^{\frac{1}{2}}}{\| f \|^{\frac{1}{2}}} $ yields
	\begin{equation}
		\begin{aligned}
			|f|_\infty \lesssim \| f \|^{\frac{1}{2}} \cdot \| \nabla^{2} f \|^{\frac{1}{2}},
		\end{aligned}
	\end{equation}
and hence the lemma follows.

\end{proof}

\newpage

%%%%%%%%%%%%%%%%%%%%%%%%%%%%%%
%\section*{Acknowledgement}
%Yi Zhou was supported by Key Laboratory of Mathematics for Nonlinear Sciences (Fudan University), Ministry of Education of China, P.R.China. Shanghai Key Laboratory for Contemporary Applied Mathematics, School of Mathematical Sciences, Fudan University, P.R. China, and by NCFC (grant No. 12171097), Shanghai Science and Technology Program (grant No. 21JC14006000).

\bibliographystyle{plain}
\bibliography{reference}

\end{CJK*}

\end{document}